\documentclass[12pt,reqno]{amsart}
\usepackage{moreverb}

\usepackage{epsfig,color}
\usepackage{amsmath,amssymb,amsfonts}
\usepackage{subfigure,psfrag}
\usepackage{ifthen}
\usepackage{epstopdf}
\usepackage{fullpage}

\definecolor{darkgreen}{rgb}{0,0.5,0}

\newtheorem{theorem}{Theorem}
\newtheorem{lemma}[theorem]{Lemma}

\newtheorem{remark}[theorem]{Remark}

\def\R{\mathbb{R}} 
\def\set#1#2{\left\{#1\,\big|\,#2\right\}}
\newcommand{\norm}[3][]{#1\|#2#1\|_{#3}}

\def\Lop#1#2{L\big(#1;#2\big)}

\newcommand{\spl}[3]{(#1,#2)_{#3}}
\newcommand{\spe}[3]{\langle#1,#2\rangle_{#3}}
\newcommand{\BBstab}[2]{\widetilde{\BB}\left(#1;#2\right)} 
\newcommand{\BBVstab}[2]{\widetilde{\BB}_V\left(#1;#2\right)} 

\def\I{\mathbf{I}} 

\def\A{\mathbf{A}} 
\def\b{\mathbf{b}}
\def\r{c}
\def\u{\mathbf{u}} 
\def\v{\mathbf{v}} %
\def\w{\mathbf{w}} 

\def\TT{\mathcal{T}} 
\def\NN{\mathcal{N}} 

\def\EE{\mathcal{E}} 
\def\EEr{\EE_{\Gamma}} 
\def\EEt{\EE_T} 

\def\AA{\mathcal{A}}
\def\BB{\mathcal{B}}
\def\HH{\mathcal{H}}
\def\OO{\mathcal{O}}

\def\normal{\mathbf{n}}

\def\CC{\mathcal{C}} 
\def\PP{\mathcal{P}} 

\def\VV{\mathcal{V}} 
\def\KK{\mathcal{K}} 
\def\SS{\mathcal{S}} 

\def\II{\mathcal{I}} 
\def\IIh{\II_h} 

\def\diam{{\operatorname{diam}}}
\def\div{\operatorname{div}}
\def\loc{{\operatorname{\ell oc}}}

\title[A non-symmetric FVM-BEM coupling]{A non-symmetric coupling of the 
finite volume method and the boundary element method}

\author{Christoph Erath}
\address{TU Darmstadt, Department of Mathematics, Dolivostra\ss{}e 15, 64293 Darmstadt, Germany}
\email{erath@mathematik.tu-darmstadt.de\quad\rm(corresponding author)}

\thanks{C. Erath (corresponding author): TU Darmstadt, Germany; erath@mathematik.tu-darmstadt.de}
\thanks{G. Of: TU Graz, Austria; of@tugraz.at}
\thanks{F.-J. Sayas: University of Delaware, USA; fjsayas@math.udel.edu}

\author{G\"unther Of}
\address{TU Graz, Steyrergasse 30, 8010 Graz, Austria}
\email{of@tugraz.at}

\author{Francisco-Javier Sayas}
\address{University of Delaware,
		 532 Ewing Hall Newark, DE 19716, USA}
\email{fjsayas@math.udel.edu}

\date{\bf\today}
\begin{document}

\begin{abstract}
As model problem we consider the prototype for flow and transport
of a concentration in porous media in an interior domain and couple 
it with a diffusion process in the corresponding unbounded exterior domain. 
To solve the problem
we develop a new non-symmetric coupling between the vertex-centered finite volume 
and boundary element method. This discretization provides naturally 
conservation of local fluxes
and with an upwind option also stability in the convection dominated case. 
We aim to provide a first rigorous analysis
of the system for different model parameters; stability, convergence, 
and a~priori estimates. This includes the use of
an implicit stabilization, known from the finite element and boundary element
method coupling.
Some numerical experiments conclude the work and confirm the theoretical
results.\\[0.5\baselineskip]
\end{abstract}

\maketitle
\section{Model problem and introduction}
Throughout this work,
let $\Omega\subset \R^d$, $d=2,3$, be a bounded domain with 
connected polygonal Lipschitz boundary $\Gamma$ and 
$\Omega_e=\R^d\backslash\overline{\Omega}$ is the corresponding
unbounded exterior domain. 
We consider the same model problem as in~\cite{Erath:2012-1,Erath:2013-2}:
find $u$ and 
$u_e$ such that
\begin{subequations}
\label{eq:model}
\begin{alignat}{2}
\label{eq1:model}
   \div (-\A \nabla u + \b u)+\r u  &= f \quad & &\text{in }\Omega,\\
\label{eq2:model}
-\Delta u_e &=0 \quad & &\text{in } \Omega_e,\\
\label{eq3:model}
u_e(x)&=C_{\infty}\log|x|+\OO(1/|x|) \quad& &\text{for }|x|\to \infty,\quad d=2,\\
\label{eq3a:model}
u_e(x)&=\OO(1/|x|) \quad& &\text{for }|x|\to \infty,\quad d=3,\\
\label{eq4:model}
 u&=u_e+u_0 \quad & &\text{on } \Gamma, \\
 \label{eq5:model}
 ( \A \nabla u-\b u)\cdot\normal&=
 \frac{\partial u_e}{\partial \normal}+t_0\quad & &\text{on }  
 \Gamma^{in},\\
 \label{eq6:model}
 ( \A \nabla u)\cdot\normal&=
 \frac{\partial u_e}{\partial \normal}+t_0\quad & &\text{on } \Gamma^{out},
\end{alignat}
\end{subequations}
where $\A$ is a symmetric diffusion matrix, 
$\b$ is a possibly dominating velocity field, 
$\r$ is a reaction function, $f$ is a source term, and $C_\infty$ is 
an unknown constant. The coefficients are allowed to be variable.
The coupling boundary 
$\Gamma=\partial \Omega= \partial\Omega_e$ is divided in an inflow and outflow part, 
namely $\Gamma^{in}:=\set{x\in\Gamma}{\b(x)\cdot\normal(x)<0}$ and
$\Gamma^{out}:=\set{x\in\Gamma}{\b(x)\cdot\normal(x)\geq 0}$, respectively, 
where $\normal$ is the normal vector on $\Gamma$ pointing outward with respect
to $\Omega$. We allow prescribed jumps $u_0$ and $t_0$ on $\Gamma$. The radiation condition for the two dimensional case, which will be complemented with the additional hypothesis that the diameter of $\Omega$ is less than one, guarantees that our problem has a unique solution. Other radiation conditions are also possible, but some lead to restrictions on the data. Changing from one to the other is a relatively simple exercise adding sources. See~\cite{McLean:2000-book,Costabel:1985-1} for more information on radiation conditions.

The model problem in the interior domain $\Omega$ is the prototype
for flow and transport of a concentration in porous media.
Usually, boundary values such as Dirichlet and/or Neumann boundary 
conditions are needed to solve the problem.
These problems are often convection dominated and the conservation law, e.g., 
local conservation of fluxes, should also be preserved for a numerical approximation
of the solution.
Therefore, a finite volume
method (FVM) is often the method of choice since they provide an easy option 
to stabilize the convection term and they natural preserve conservation of
numerical fluxes due to
their formulation. However, if the domain is unbounded one would have to truncate
the domain. The above formulation solves also another issue, i.e., if we do
not know any boundary conditions, we assume a diffusion process in the
corresponding (unbounded) exterior domain $\Omega_e$, which ``replaces''
the boundary values.
The method of choice for unbounded domains is the boundary element method (BEM) 
which reduces the discretization to the boundary and therefore avoids
the truncation of $\Omega_e$.
Therefore, we consider 
an FVM-BEM coupling as in~\cite{Erath:2010-phd,Erath:2012-1,Erath:2013-2}.
To the best of the authors knowledge, these works are the 
first theoretical justifications 
of a FVM-BEM coupling, where  
a three field coupling
approach is used with either the vertex-centered 
(finite volume element method, box method) FVM or the 
cell-centered FVM. 

In this work we analyze and verify a non-symmetric FVM-BEM coupling with
the vertex-centered FVM, in the following only named FVM. The main 
motivation of using this is to get an easier coupling formulation and a smaller
system of linear equations, which saves computational costs.
The idea of a non-symmetric coupling approach goes back
to~\cite{Johnson:1980-1,Brezzi:1979-1}. This
coupling formulation applied for a finite element method (FEM)-BEM
discretization is also known as
Johnson-N{\'e}d{\'e}lec coupling.
However, the analysis in this early works relied on 
specific choices of the discretization spaces or on
the compactness of a certain 
integral operator, which was in fact a restriction to a smooth boundary.
In particular, a rigorous
mathematical analysis for Lipschitz domains was not known. 
Recently, the work in~\cite{Sayas:2009-1} provided a first analysis, which
overcame these restrictions. Meanwhile, several extensions and simplifications
are possible, such that a SIAM review paper~\cite{Sayas:review} was published.
Among these extensions there are results on the non-symmetric formulation for
the potential equation with variable coefficients~\cite{Of:2013-1,Steinbach}, 
non-linearities~\cite{wien,Dirk:Ela}, for 
elasticity~\cite{Dirk:Ela,Steinbach:ela}, and for boundary value 
problems~\cite{Sayas:relaxing,Of:Drwp}.   
In addition, similar results have been reported on related coupling
formulations~\cite{wien,Sayas:relaxing} and the DG-BEM coupling~\cite{Heuer:DG}. 
We want to mention that the counterpart to the non-symmetric coupling
is the so called the symmetric coupling first introduces in~\cite{Costabel:1987-1}.
However, symmetry is referred to a diffusion--diffusion transmission problem, i.e.,
the whole system is symmetric.
We stress that this would be destroyed if 
one applies convection in the interior domain.

There exist a couple of papers, which analyze the vertex-centered 
FVM, e.g.,~\cite{Bank:1987-1,Hackbusch:1989-1} to mention only the very first
works. It is well known that
for pure diffusion with piecewise constant diffusion coefficient on a primal mesh
the standard FEM and the FVM bilinear form are exactly the same. Thus the schemes
differ basically only on the right-hand side. However, for all other  
diffusion problems~\cite{Cai:1991-1} 
and a possible convection field and a reaction term the systems are different. 
Contrary to standard FEM, FVM still provides local flux conservation due to 
its formulation and provides an easy upwind stability option for convection dominated
problems.
The standard analysis approach makes use of a comparison between the
FEM and FVM bilinear 
form~\cite{Bank:1987-1,Hackbusch:1989-1,Cai:1991-1,Ewing:2002-1,Chatzipantelidis:2002-1}.
For our FVM-BEM coupling we may apply similar techniques for the FVM part.
Note that contrary to a classical FEM-BEM coupling we do not have a
classical Galerkin orthogonality property due to the FVM formulation based on the
conservation law. Thus the analysis differs significantly to an FEM-BEM analysis.
However, we use the equivalent formulation of a stabilized continuous 
coupling formulation, extended here for the convection-diffusion-reaction problem in $\Omega$,
and compute an ellipticity constant.
Based on the continuous stabilization we introduce a stabilization for the 
FVM-BEM coupling.
This is needed for pure diffusion models and for convection-diffusion-reaction problems, where
the energy norm reduced to a semi-norm. 
We stress that the stabilization is only needed for theoretical purposes since the formulation
is equivalent to the standard system. 
We aim to provide a discrete ellipticity estimate, convergence, and a~priori estimates for the 
FVM-BEM coupling.
Our new analysis technique
gives us a recipe
for the coupling of BEM with a non-Galerkin method like FVM.
Furthermore, this work improves the 
results in~\cite{Erath:2010-phd,Erath:2012-1} for a three field FVM-BEM coupling, 
where we had
to assume a little bit more regularity on the unknown exterior conormal
solution and some constraints on the convection and reaction terms
for some special model problem configurations. 
However, as for the non-symmetric FEM-BEM coupling we have a theoretical constraint
on the eigenvalues of $\A$, which is not needed in~\cite{Erath:2010-phd,Erath:2012-1}.
 
Throughout, we denote by $L^m(\cdot)$ and $H^m(\cdot)$, $m>0$ the standard Lebesgue 
and Sobolev spaces equipped with the usual norms $\norm{\cdot}{L^2(\cdot)}$ and
$\norm{\cdot}{H^m(\cdot)}$, respectively. For $\omega\subset\Omega$,
$\spl{\cdot}{\cdot}{\omega}$ is the $L^2$ scalar product. The space $H^{m-1/2}(\Gamma)$
is the space of all traces of functions from $H^m(\Omega)$ and the duality
between $H^{m}(\Gamma)$ and $H^{-m}(\Gamma)$ is given by the extended 
$L^2$-scalar product $\spe{\cdot}{\cdot}{\Gamma}$. 
The space $H^1_{\loc}(\Omega):=\set{v:\Omega\to\R}{v|_K\in H^1(K), \, \text{for all } K\subset \Omega \,\mbox{open and bounded}}$ collects functions with local $H^1$ behavior. Furthermore, the Sobolev space $W^{1,\infty}$ contains exactly the Lipschitz 
continuous functions. If it is clear from the context, we do not use a notational 
difference for functions in a domain and its traces.
To simplify the presentation we equip the 
space $\HH:=H^1(\Omega)\times H^{-1/2}(\Gamma)$
with the norm
\begin{align*}
	\norm{\v}{\HH}^2:=\norm{v}{H^1(\Omega)}^2+\norm{\psi}{H^{-1/2}(\Gamma)}^2
\end{align*}
for $\v=(v,\psi)\in\HH$.

With this notation we can specify the model data as:
the diffusion matrix $\A:\Omega\to\R^{d\times d}$ has entries in 
$W^{1,\infty}(\Omega)$,
is bounded, symmetric and uniformly positive definite, i.e.,
there exist positive constants $C_{\A,1}$ and $C_{\A,2}$
with
$C_{\A,1}|\mathbf{v}|^2\leq \mathbf{v}^T\A(x)\mathbf{v}
	        \leq C_{\A,2}|\mathbf{v}|^2$
for all $\mathbf{v}\in \R^d$ and almost every $x\in\Omega$.
We will also admit coefficients $\A$ that are $\TT$-piecewise constant, where $\TT$ denotes the triangulation 
of $\Omega$ introduced in subsection~\ref{subsec:triangulation}, satisfying identical 
symmetry and uniform positive definiteness assumptions. Note that the best constant $C_{\A,1}$ equals the infimum over $x\in \Omega$ of the minimum eigenvalue of $\A(x)$, 
which we will denote $\lambda_{\min}(\A)$.
Furthermore, $\b\in W^{1,\infty}(\Omega)^d$
and $\r\in L^{\infty}(\Omega)$
satisfy 
\begin{align}
  \label{eq:bcestimate1}
  \gamma(x):=\frac{1}{2}\div\b(x)+\r(x), \qquad\gamma(x)\geq 0\quad
  \text{for almost every }x\in\Omega
\end{align}
with the function $\gamma\in L^{\infty}(\Omega)$.
We stress that our analysis holds
for constant $\b$ and $\r=0$ as well.
Finally, we choose the right-hand side $f\in L^2(\Omega)$, $u_0\in H^{1/2}(\Gamma)$,
and $t_0\in H^{-1/2}(\Gamma)$. In the two dimensional
case we additionally assume $\diam(\Omega)<1$ to ensure $H^{-1/2}(\Gamma)$ ellipticity 
of the single layer operator defined below.

Then our model problem reads in a weak sense: find $u\in H^1(\Omega)$ and 
$u_e \in H^1_{\loc}(\Omega_e)$ such that~\eqref{eq1:model}--\eqref{eq6:model}
hold.

The model problem~\eqref{eq:model} admits a unique solution for both, the 
two and three dimensional case~\cite{Erath:2012-1}.
\begin{remark}
To replace the radiation condition~\eqref{eq3:model} by
$u_e(x)=\OO(1/|x|)$ for $|x|\to \infty$ in two dimensions
one would have to assume the
the scaling condition
\begin{align*}
	\spe{\partial u_e/\partial\normal}{1}{\Gamma}=0
\end{align*}
to guarantee solvability. As opposed for the purely diffusive case, 
this condition cannot be easily transformed into a condition on the data.	
\end{remark}

The content of this paper is organized as follows. 
Section $2$ gives a short summary on integral equations and
the weak formulation of our model problem based on the non-symmetric
approach. We show an ellipticity estimate 
through an equivalent stabilized weak formulation 
and state the ellipticity constant explicitly. 
In section $3$ we introduce the non-symmetric FVM-BEM coupling to solve our 
model problem.
Section $4$ proves stability, convergence, and an a~priori result for 
our coupling.
Numerical experiments, found in section $5$,
confirm the theoretical results. Some conclusions complete to work.
%
\section{Integral equation and weak coupling formulation}
The representation formula for the exterior Laplace equation~\eqref{eq2:model} with the
radiation condition~\eqref{eq3:model}-\eqref{eq3a:model} and 
$\phi(x)=\frac{\partial}{\partial \normal}u_e(x)|_\Gamma$, $x\in\R$
reads
\begin{align}
  \label{eq:repformula}
  u_e(x)=-\int_\Gamma G(x-y)\phi(y)\,ds_y
         +\int_\Gamma \frac{\partial}{\partial\normal_y}G(x-y)u_e(y)|_\Gamma\,ds_y
\end{align}
with the fundamental solution for the Laplace operator
\begin{align*}
  G(z):=\begin{cases}\displaystyle
  -\frac{1}{2\pi}\log|z|\quad&\text{for }z\in\R^2\backslash\{0\},\\[2ex]
\displaystyle
  \frac{1}{4\pi}\frac{1}{|z|}&\text{for }z\in\R^3\backslash\{0\}.
  \end{cases}
\end{align*}
From~\eqref{eq:repformula} we obtain (taking traces) the boundary integral 
equation on $\Gamma$
\begin{align}
\label{eq:cal1}
u_e|_\Gamma=(1/2+\KK)u_e|_\Gamma-\VV\phi.
\end{align}
The single layer operator $\VV$ and the double layer
operator $\KK$ are given, for smooth enough input, by
\begin{align*}
(\VV \psi)(x)=\int_{\Gamma}\psi(y)G(x-y)\,ds_{y}
 	\qquad
(\KK \theta)(x)=
  \int_{\Gamma}\theta(y)\frac{\partial}{\partial 
  \normal_{y}}G(x-y)\,ds_{y} \qquad x\in \Gamma,
\end{align*}
where $\normal_y$ is a normal vector with respect to $y$.
The integral equation~\eqref{eq:cal1} holds on $\Gamma$ except on corners and edges. 
We recall~\cite[Theorem 1]{Costabel:1988-1} that  these operators can be extended to bounded operators
\begin{align*}
\VV \in\Lop{H^{s-1/2}(\Gamma)}{H^{s+1/2}(\Gamma)},
	\qquad
\KK \in\Lop{H^{s+1/2}(\Gamma)}{H^{s+1/2}(\Gamma)},
	\quad s\in [-\tfrac12,\tfrac12].  
\end{align*}
It is also well-known that $\VV$ is symmetric and $H^{-1/2}(\Gamma)$ elliptic, since 
we additionally assume $\diam(\Omega)<1$ in the two dimensional case, 
which can always be achieved by scaling. The expressions
\begin{align*}
\norm{\cdot}{\VV}^2:=\spe{\VV\cdot}{\cdot}{\Gamma},
	\qquad
\norm{\cdot}{\VV^{-1}}^2:=\spe{\cdot}{\VV^{-1}\cdot}{\Gamma}
\end{align*}
define norms in $H^{-1/2}(\Gamma)$ and $H^{1/2}(\Gamma)$, respectively. These norms are equivalent to the usual ones.

We consider a weak form of the model problem~\eqref{eq:model} in terms of boundary integral
operators. For that we use the non-symmetric approach, i.e,
calculate the weak formulation of the interior problem and replace the interior
conormal derivative by the exterior $\phi:=\partial u_e/\partial\normal|_\Gamma$ 
and the corresponding jump relations $t_0$,~\eqref{eq5:model}--\eqref{eq6:model}. Second,
we take the weak form of~\eqref{eq:cal1} and replace the exterior trace
$u_e|_\Gamma$ by the interior trace $u_{|\Gamma}$ and the jump $u_0$,~\eqref{eq4:model}. Then  
the coupling reads:
find 
$u\in H^1(\Omega)$, $\phi\in H^{-1/2}(\Gamma)$
such that 
\begin{subequations}
\begin{align}
\label{eq1:weakfem}
   \AA(u,v)-\spe{\phi}{v}{\Gamma}
   &= \spl{f}{v}{\Omega}+\spe{t_0}{v}{\Gamma},\\
   \label{eq2:weakfem} 
   \spe{\psi}{(1/2-\KK)u}{\Gamma}+\spe{\psi}{\VV\phi}{\Gamma}&=\spe{\psi}{(1/2-\KK)u_0}{\Gamma}
\end{align}
\end{subequations}
for all 
$v\in H^1(\Omega)$, $\psi\in H^{-1/2}(\Gamma)$. The bilinear form in~\eqref{eq1:weakfem} is given by
\begin{align*}
  \AA(u,v):=\spl{\A\nabla u-\b u}{\nabla v}{\Omega}+\spl{\r u}{v}{\Omega}
+\spe{\b\cdot\normal \,u}{v}{\Gamma^{out}}.
\end{align*}
%
%
%
%
\begin{lemma}
  \label{lem:coerccont}
  The bilinear form $\AA$ is coercive and 
  continuous on $H^1(\Omega)\times H^1(\Omega)$, i.e.,
  for all $v,w\in H^1(\Omega)$ and $\gamma(x)$ from assumption~\eqref{eq:bcestimate1}
  there holds
  \begin{align}
  \label{eq:bilinearintcoerciv}
  \AA(v,v)&\geq 
  \begin{cases}\displaystyle
    C_{\AA,1}\norm{v}{H^1(\Omega)}^2\quad&
	\text{for }\gamma(x)>0\text{ almost everywhere in }\Omega,\\[1mm]
    C_{\AA,1}^\star\norm{v}{H^1(\Omega)}^2\quad&
	\text{for }\gamma(x)>0\text{ on }\omega\subsetneq\Omega,|\omega|>0, \gamma(x)=0 \text{ elsewhere},\\[1mm]
    C_{\AA,1}'\norm{\nabla v}{L^2(\Omega)}^2&
  	\text{for }\gamma(x)=0\text{ almost everywhere in } \Omega,
    \end{cases}	\\[2mm]
  \label{eq:bilinearintcont}
    |\AA(w,v)|&\leq C_{\AA,2}\norm{w}{H^1(\Omega)}\norm{v}{H^1(\Omega)}.
  \end{align}
  Here, the constants
  $C_{\AA,1}=\min\{\lambda_{\min}(\A),\inf_{x\in\Omega}\gamma(x)\}>0$, $C_{\AA,1}'=\lambda_{\min}(\A)>0$
  and $C_{\AA,2}>0$, 
  depend on the data $\A$, $\b$ and $\r$.
  The constant $C_{\AA,1}^\star=\min\{\lambda_{\min}(\A),C(\gamma(x),\omega,\Omega)\}>0$
  depends additionally on the constant $C(\gamma(x),\omega,\Omega)>0$, which 
  is not known but depends
  on $\gamma(x)>0$ in $\omega$, $\omega$, and $\Omega$.
\end{lemma}
\begin{proof}
There holds
\begin{align*}
  \int_{\Gamma^{out}}\b\cdot\normal \,v^2\,ds\geq
    \frac{1}{2}\int_{\Gamma}\b\cdot\normal\, v^2\,ds
    =\frac{1}{2}\int_{\Omega}\div(\b v^2)\,dx=
    \frac{1}{2}\spl{(\div \b) v}{v}{\Omega}+\spl{\b v}{\nabla v}{\Omega}.
  \end{align*}
If $\frac12\div\b(x)+c(x)\geq\gamma(x)>0$ of assumption~\eqref{eq:bcestimate1} is positive
almost everywhere in $\Omega$, it follows that
  \begin{align*}
    \AA(v,v)&\geq\spl{\A\nabla v}{\nabla v}{\Omega}+
    \frac{1}{2}\spl{(\div \b) v}{v}{\Omega}+
    \spl{\r v}{v}{\Omega}\geq C_{\AA,1}\norm{v}{H^1(\Omega)}^2.
  \end{align*}

If
$\gamma(x)>0$ holds on a set $\omega\subsetneq\Omega$ 
of positive measure but $\gamma(x)=0$ on $\Omega\backslash \omega$, 
we can use a compactness argument (or the Deny-Lions theorem) to 
prove coercivity of $\AA$ in $H^1(\Omega)$. 
Then the coercivity constant $C_{\AA,1}^\star$ is not known.   
When $\gamma(x)=0$ almost everywhere in $\Omega$, 
we only obtain coercivity of $\AA$ with respect to the $H^1$ seminorm
and the constant $C_{\AA,1}'$. 
Using simple arguments, the continuity bound~\eqref{eq:bilinearintcont} can be easily proved with
\begin{align*}
C_{\AA,2}=2\max\{ \norm{\A}{L^\infty(\Omega)^{d\times d}}, 
					\norm{\b}{L^\infty(\Omega)},
					\norm{\r}{L^\infty(\Omega)}\}
		+ C_\Gamma^2 \norm{\b\cdot\normal}{L^\infty(\Gamma^{out})},
\end{align*}
where $C_\Gamma$ is the norm of the the trace operator $H^1(\Omega)\to L^2(\Gamma^{out})$.
\end{proof}
For convenience the system~\eqref{eq1:weakfem}-\eqref{eq2:weakfem} can be written in the product space $\HH=H^1(\Omega)\times H^{-1/2}(\Gamma)$ as follows: we introduce the bilinear form $\BB:\HH\times\HH\to \R$
\begin{align}
	\label{eq:Bfem}
	\BB((u,\phi);(v,\psi))&:=   \AA(u,v)-\spe{\phi}{v}{\Gamma}
      +\spe{\psi}{(1/2-\KK)u}{\Gamma}+\spe{\psi}{\VV\phi}{\Gamma},
\end{align}
and the linear functional
\begin{align}
	\label{eq:Ffem}
	F((v,\psi)):=  \spl{f}{v}{\Omega} +\spe{t_0}{v}{\Gamma}+\spe{\psi}{(1/2-\KK)u_0}{\Gamma}.  
\end{align}
Then~\eqref{eq1:weakfem}-\eqref{eq2:weakfem} is equivalent to: find
$\u\in\HH$ such that
\begin{align}
	\label{eq:FEMBEM}
	\BB(\u;\v)=F(\v) \qquad \text{for all } \v\in \HH.
\end{align}
With integration by parts we calculate
\begin{align*}
	\BB(\v;\v)&=   
	\spl{\A\nabla v}{\nabla v}{\Omega}+\spl{\big(\tfrac12\div\b+\r\big) v}{v}{\Omega}
	-\spe{\b\cdot\normal\, v}{v}{\Gamma^{in}}+\spe{\b\cdot\normal\, v}{v}{\Gamma^{out}}\\
	&\qquad-\spe{\psi}{v}{\Gamma}
      +\spe{\psi}{(1/2-\KK)v}{\Gamma}+\spe{\psi}{\VV\psi}{\Gamma},	
\end{align*}
and thus we see
\begin{align*}
\BB((1,0);(1,0))=\int_\Omega \big(\tfrac12 \div\b+\r\big)+\int_\Gamma |\b\cdot\normal|.
\end{align*}
Thus if $\frac{1}{2}\div\b+\r=0$ in $\Omega$ and $\b\cdot\normal=0$ on $\Gamma$ (in particular, 
when $\b=(0,0)^T$ and $\r=0$), it follows that $\BB((1,0);(1,0))=0$. 
This lack of coercivity will be remedied using an equivalent variational problem 
for the sake of analysis.
 
Therefore, we define the linear operator
\begin{align*}
  P((v,\psi)) := \spe{1}{(1/2-\KK) v + \VV \psi}{\Gamma}=
  \int_\Gamma ((1/2-\KK )v+\VV \psi)
\end{align*}
and introduce a parameter $\beta$ depending on $\gamma(x)$ 
of assumption~\eqref{eq:bcestimate1};
\begin{align}
	\label{eq:beta}
  \beta&:= 
  \begin{cases}\displaystyle
    1\quad& \text{if }\gamma(x)=0\text{ almost everywhere in }\Omega,\\[1mm]
    0\quad& \text{else}.
    \end{cases}	
\end{align}
Then the $\beta$-dependent perturbations of the bilinear form $\BB(\u,\v)$ is
\begin{equation}\label{bbtilde}
\BBstab\u\v:=\BB(\u,\v)+\beta P(\u)\,P(\v),
\end{equation}
and of the linear map $F(\v)$
\begin{equation}\label{ftilde}
\widetilde{F}(\v):=F(\v)+\beta \spe{1}{(1/2-\KK) u_0}\Gamma P(\v).
\end{equation}
Thus a stabilized variational formulation is given by: find $\u\in \HH$ such that
\begin{align}
	\label{eq:FEMBEMstab}
	\BBstab{\u}{\v}=\widetilde{F}(\v) \qquad\text{for all } \v\in \HH.
\end{align}
Note that this type of stabilization has also been 
considered in~\cite{Of:2013-1} and~\cite{wien}. 
We emphasize that this formulation is introduced purely for 
theoretical purposes, and the discretization will be applied 
directly on~\eqref{eq1:weakfem}-\eqref{eq2:weakfem}.
\begin{lemma}\label{lem:equi:fem}
  The variational formulation~\eqref{eq:FEMBEM} and the stabilized version
  in~\eqref{eq:FEMBEMstab} are equivalent.
\end{lemma}
\begin{proof}
The equivalence of formulations was stated in~\cite[Theorem 14]{wien} for
a pure diffusion problem. The convection and reaction terms in the bilinear
form $\AA(\cdot,\cdot)$ do not affect the proof. 
We note that we will see a similar result for the FVM-BEM discretization
in Lemma~\ref{lem:equi:fvm}. 
\end{proof}
The next theorem on the coercivity of the bilinear form $\widetilde{\BB}$
is an extended and improved version of the one stated in~\cite[Theorem~3.1]{Of:2013-1} 
and~\cite[Theorem~15]{wien} for a purely diffusive problem. 
We extend it by the convection and reaction terms in the bilinear form and 
present an improved ellipticity constant compared
to~\cite[Theorem~3.1]{Of:2013-1}. This is possible due to some modification of the
proof inspired by~\cite{Of:Drwp}. Before we state the theorem, we 
recall an important contractivity result for 
the double layer operator~\cite[Lemma 2.1]{Of:2013-1}
with the contraction constant $C_{\KK}$ from~\cite{Steinbach:2001-1}: 
there exists $C_{\KK}\in[1/2,1)$ such that
\begin{align}\label{boundcK}
  \norm{(1/2+\KK)v}{\VV^{-1}}^2
  \leq  C_{\KK} \spe{\VV^{-1}(1/2+\KK) v}{v}{\Gamma}.
\end{align}
Furthermore, we define
for $\beta=0$
 \begin{align}
  \label{eq:Cbc}
  C_{\b\r}&:= 
  \begin{cases}\displaystyle
    \inf_{x\in\Omega}\gamma(x)\quad&
	\text{for }\gamma(x)>0\text{ almost everywhere in }\Omega,\\[1mm]
    C(\gamma(x),\omega,\Omega)\quad&
	\text{for }\gamma(x)>0\text{ on }
	\omega\subsetneq\Omega,|\omega|>0, \gamma(x)=0 \text{ elsewhere}
    \end{cases}	
  \end{align}
  with $\gamma(x)$ from assumption~\eqref{eq:bcestimate1}
  and the unknown constant $C(\gamma(x),\omega,\Omega)>0$
  introduced in Lemma~\ref{lem:coerccont}.

\begin{theorem}
 If $\lambda_{\min}(\A)>C_{\KK}/4$, then 
$\widetilde{\BB}$ is $\HH$-elliptic. More precisely, for all $\v=(v,\psi) \in\HH$ holds
  \begin{align}\label{eq:ellipticity:konv}
    \BBstab{\v}{\v} 
    \geq C_{\rm stab}
    \left[ \norm{\nabla v}{L_2(\Omega)}^2
      + (1-\beta) \norm{v}{L_2(\Omega)}^2
      + \beta P(\v)^2 + \norm{\psi}{\VV}^2 \right].
  \end{align}
The stability constant $C_{\rm stab}$ reads 
\begin{align*}
    C_{\rm stab}  = 
    \begin{cases}
      \min  \left\{ C_{\b\r}, \frac{1}{2}  \left[  
          \lambda_{\min}(\A) +1  
          - \sqrt{ (\lambda_{\min}(\A)-1)^2+C_{\KK}}
        \right] \right\} 
      & \text{for } \beta = 0,\\[2ex]
      \min  \left\{ 1, \frac{1}{2}  \left[  
          \lambda_{\min}(\A) +1  
          - \sqrt{ (\lambda_{\min}(\A)-1)^2+C_{\KK}}
        \right] \right\} 
      & \text{for } \beta = 1
\end{cases}
\end{align*}  
and depends on the model data $\A$, $\b$, $\r$, and the contraction constant
$C_{\KK}$.
\end{theorem}
\begin{remark}
  The right-hand side in~\eqref{eq:ellipticity:konv} defines an equivalent
  norm in $\HH$. While this is obvious for $\beta=0$, a simple compactness argument
   (see~\cite[Lemma 10 and (65)]{wien} for a similar argument) shows the equivalence for $\beta=1$.
  Note that we only do not know the constant $C_{\rm stab}$ explicitly
  in the second case of~\eqref{eq:Cbc}.
 \end{remark}
\begin{proof}
  The proof is in the spirit of previous 
  publications~\cite{Of:2013-1,Of:Drwp,wien} on the non-symmetric 
  FEM-BEM coupling, but extended here for the different interior model
  problem. Therefore, we only present the key points.

An element $v\in H^1(\Omega)$ can be decomposed as a sum $v=v_\Gamma+v_0$, where $v_\Gamma$ is harmonic and $v_0\in H^1_0(\Omega)$. Thus
$(\nabla v_\Gamma,\nabla w)_\Omega=0$ for all  $w\in H^1_0(\Omega)$,
which implies that
\begin{align}
\norm{\nabla v}{L^2(\Omega)}^2
&= \norm{\nabla v_0}{L^2(\Omega)}^2+\norm{\nabla v_\Gamma}{L^2(\Omega)}^2
 = \norm{\nabla v_0}{L^2(\Omega)}^2 + \spe{S^{\text{int}} v}{v}{\Gamma},
\label{eq:split}
\end{align}
where $S^{\text{int}}:=\VV^{-1}(1/2+\KK)$ 
denotes the Steklov--Poincar\'e operator, i.e., the Dirichlet to Neumann
map of the interior Laplace problem. The term 
$\spe{S^{\text{int}}v}{v}{\Gamma}$ will help to compensate possible negative
contributions of the non-symmetric coupling to the total energy of the system.
Let us first recall our choice of $\beta$ 
depending on $\gamma(x)$ in~\eqref{eq:beta} 
and the definition of $C_{\b\r}$ in~\eqref{eq:Cbc}. 
This allows us to write the coercivity estimate of Lemma~\ref{lem:coerccont} as
\begin{align*}
\AA(v,v)\ge  \lambda_{\min}(\A) \norm{\nabla v}{L^2(\Omega)}^2
  + (1-\beta) C_{\b\r} \norm{v}{L_2(\Omega)}^2 \qquad\text{for all } v\in H^1(\Omega).
\end{align*}
Following~\cite{Of:2013-1} and using~\eqref{boundcK}, we can easily estimate
\begin{align*}
  \spe{\psi}{(1/2+\KK)v}{\Gamma}
  &= \spe{\VV \psi}{ \VV^{-1}(1/2+\KK)v}{\Gamma}\\
  &\leq  \norm{\VV^{-1}
(1/2+\KK)v}{\VV} \norm{\psi}{\VV} 
  = \norm{(1/2+\KK)v}{\VV^{-1}} \norm{\psi}{\VV} \\
  &\leq C_{\KK}^{1/2}\spe{S^{\text{int}} v}{v}{\Gamma}^{1/2}\norm{\psi}{\VV}
\end{align*}
for all $(v,\psi)\in \HH$. Therefore, for all $\v=(v,\psi)\in \HH$, we can estimate
\begin{align*}
  \BBstab{\v}{\v}  &=\ 
  \AA(v,v)
  + \spe{\psi}{\VV\psi}{\Gamma}
  - \spe{\psi}{(1/2+\KK)u}{\Gamma}+ \beta P(\v)^2\\  
  &\geq
  \lambda_{\min}(\A) \norm{\nabla v}{L_2(\Omega)}^2
  + (1-\beta) C_{\b\r} \norm{v}{L_2(\Omega)}^2
  + \beta P(\v)^2
  + \norm{\psi}{\VV}^2
   - C_{\KK}^{1/2}\spe{S^{\text{int}} v}{v}{\Gamma}^{1/2}\norm{\psi}{\VV}\\
  & \geq
  \lambda_{\min}(\A) \norm{\nabla v_0}{L_2(\Omega)}^2
  + (1-\beta) C_{\b\r} \norm{v}{L_2(\Omega)}^2
  + \beta P(\v)^2 \\ 
   &\phantom{=} + \begin{pmatrix}
    \spe{S^\text{int} v}{v}{\Gamma}^{1/2}\\
    \norm{\psi}{\VV}
  \end{pmatrix}^\top
  \begin{pmatrix}
    \lambda_{\min}(\A) & -\frac{1}{2} \sqrt{C_{\KK}}\\
    -\frac{1}{2} \sqrt{C_{\KK}} & 1
  \end{pmatrix}
  \begin{pmatrix}
    \spe{S^\text{int} v}{v}{\Gamma}^{1/2}\\
    \norm{\psi}{\VV}
  \end{pmatrix},
\end{align*}
where in the last inequality we have used the harmonic splitting~\eqref{eq:split}.
Since $\lambda_{\min}(\A) >0$, the quadratic form in the right-hand side of the 
above estimate is positive definite if and only if
\begin{align*}
  \begin{vmatrix}
  \lambda_{\min}(\A) & -\frac{1}{2} \sqrt{C_{\KK}}\\
 -\frac{1}{2} \sqrt{C_{\KK}} & 1
\end{vmatrix}
= \lambda_{\min}(\A) - \frac{1}{4} C_{\KK} > 0.
\end{align*}
Calculating
the smallest eigenvalue of the matrix above, we can bound
\begin{align*}
  \BBstab{\v}{\v} \geq C_{\rm stab}\Big(
  &
  \norm{\nabla v_0}{L_2(\Omega)}^2
  + \spe{S^{\text{int}} v}{v}{\Gamma}
  + (1-\beta) \norm{v}{L_2(\Omega)}^2\\
  & 
  + \beta P(\v)^2 + \norm{\psi}{\VV}^2
  \Big),
\end{align*}
which, using~\eqref{eq:split}, is the estimate of the statement of the theorem. 
\end{proof}
\begin{remark}
  Note that this result also improves the estimate
  of~\cite[Theorem 3.1]{Of:2013-1} for a pure diffusion problem in $\Omega$.
  The smallest eigenvalue in $C_{\rm stab}$ in the 
  case $\beta=1$ is observed to be
  sharp in the numerical experiments  of~\cite{Of:2013-1}, contrary to the constant
  reported therein. 
\end{remark}
Using the boundedness of $\AA$ in~\eqref{eq:bilinearintcont} and 
mapping properties of the integral operators, it is easy to conclude that the bilinear form $\widetilde{\BB}$ defined in~\eqref{bbtilde} and the linear form $\widetilde F$ 
in~\eqref{ftilde} are bounded. Thus we can conclude the
unique solvability of~\eqref{eq:FEMBEMstab}. Due to the equivalence of the
formulations in Lemma~\ref{lem:equi:fem}, the original variational formulation
\eqref{eq1:weakfem}-\eqref{eq2:weakfem} is uniquely solvable. 
\begin{remark}
\label{rem:discellipticity}
Note that the equivalence of~\eqref{eq:FEMBEMstab}
and~\eqref{eq:FEMBEM} shown in Lemma~\ref{lem:equi:fem}
and the ellipticity estimate~\eqref{eq:ellipticity:konv}
also hold true on the discrete level,
if the constants are in the discretization space of $H^{-1/2}(\Gamma)$.
In other words, a possible FEM-BEM coupling solution, as shown
in Remark~\ref{rem:FEMBEM}, exists and is unique
and the C\'ea Lemma applies.
\end{remark}
%
\section{A non-symmetric FVM-BEM coupling}
In this section we develop a FVM-BEM coupling discretization in the sense of
a non-symmetric coupling approach. From now on we assume $t_0\in L^2(\Gamma)$.
First, let us introduce the notation for the triangulation and some discrete function
spaces.
\subsection{Triangulation}
\label{subsec:triangulation}
Throughout, $\TT$ denotes a triangulation or primal mesh of $\Omega$, and
$\NN$ and $\EE$ are the corresponding set of nodes and edges/faces, respectively.
The elements $T\in\TT$ are non-degenerate triangles (2-D case) or tetrahedra (3-D case),
and considered to be closed.
For the Euclidean diameter of $T\in\TT$ we write 
$h_T:=\sup_{x,y\in T}|x-y|$.
Moreover, $h_E$ denotes the length of an edge or Euclidean diameter of $E\in\EE$.
The triangulation is regular
in the sense of Ciarlet~\cite{Ciarlet:1978-book},
i.e., the ratio of the diameter $h_T$ of any element $T\in\TT$ to the diameter
of its largest inscribed ball is bounded by a constant independent
of $h_T$, the so called shape-regularity constant. 
Additionally, we assume that the triangulation $\TT$ is aligned 
with the discontinuities
of the coefficients $\A$, $\b$, and $\r$ of the differential equation (if any), the data
$f$, $u_0$, and $t_0$. Throughout, if $\normal$ appears in a boundary
integral, it denotes the unit normal vector to the boundary pointing outward the
domain.
We denote by $\EEt\subset\EE$ the set of all edges/faces of $T$, i.e.,
$\EEt:=\set{E\in\EE}{E\subset \partial T}$ and by
$\EEr:=\set{E\in\EE}{E\subset\Gamma}$ the set of 
all edges/faces on the boundary $\Gamma$.
%
%
%
%
%
\begin{figure}[t]
  \centering
  \psfrag{v1}[c][c]{\small $V_{1}$}\psfrag{v2}[c][c]{\small $V_{2}$}
  \psfrag{v3}[c][c]{\small $V_{3}$}
  \psfrag{v4}[c][c]{\small $V_{4}$}\psfrag{v5}[c][c]{\small $V_{5}$}
  \psfrag{v6}[c][c]{\small $V_{6}$}
  \psfrag{v7}[c][c]{\small $V_{7}$}
  \psfrag{a1}[c][c]{\small $a_{1}$}\psfrag{a2}[c][c]{\small $a_{2}$}
  \psfrag{a3}[c][c]{\small $a_{3}$}
  \psfrag{a4}[c][c]{\small $a_{4}$}\psfrag{a5}[c][c]{\small $a_{5}$}
  \psfrag{a6}[c][c]{\small $a_{6}$}
  \psfrag{a7}[c][c]{\small $a_{7}$}  
  \subfigure[Constructions of $\TT^*$.]
  {\label{fig:dualconst}\includegraphics[height=0.4\textwidth]{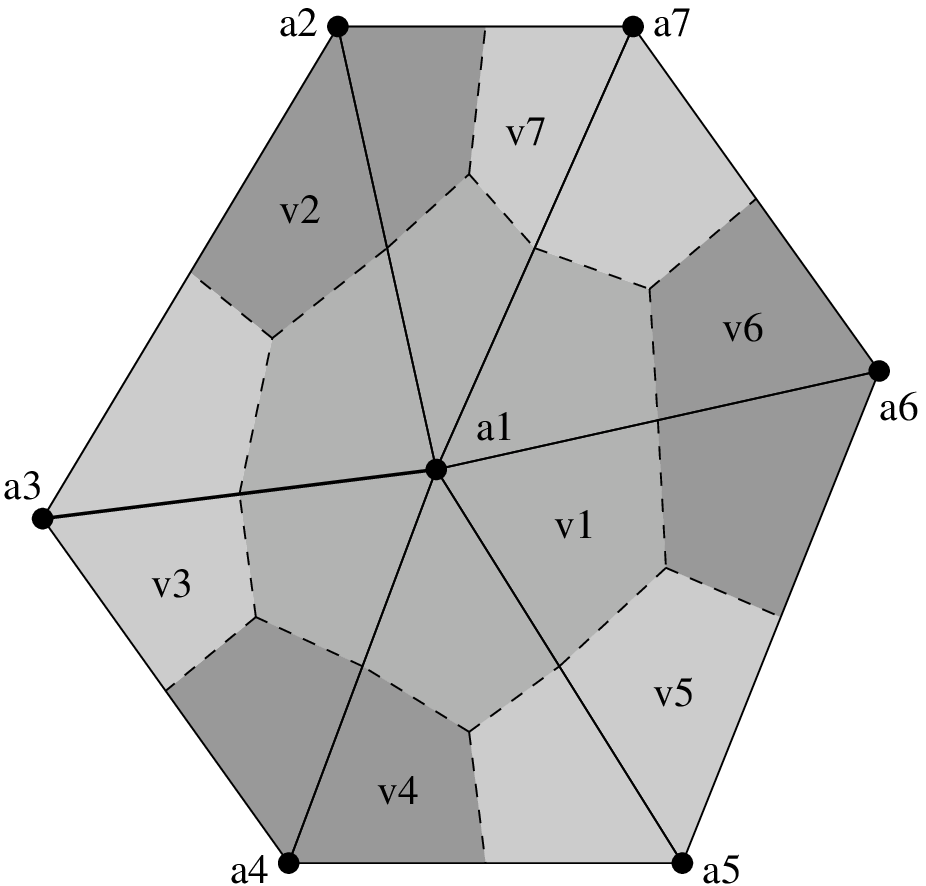}}
  \hspace*{0.1\textwidth}
  \psfrag{t17}[c][c]{\small $\tau_{17}$}
  \psfrag{t34}[c][c]{\small $\tau_{34}$}
  \subfigure[Edges for upwinding.]
  {\label{fig:edgeupwind}\includegraphics[height=0.4\textwidth]{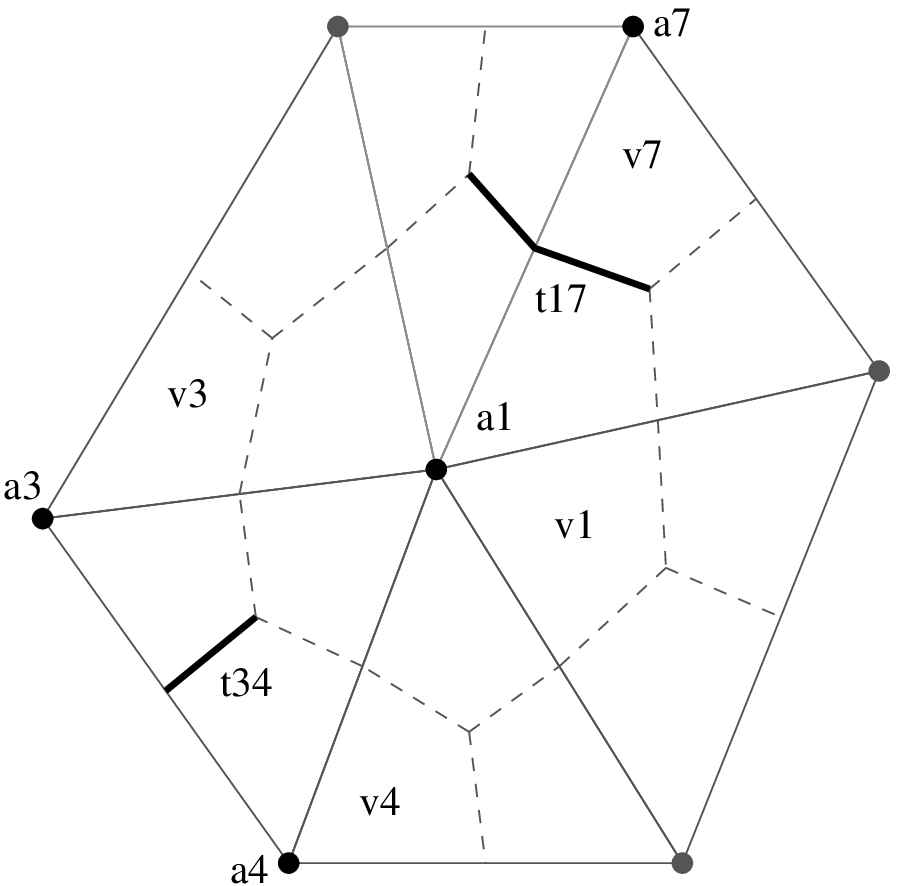}}
  \caption{
  The construction of the dual mesh $\TT^*$ from the
  primal mesh $\TT$ in two dimensions with the center of gravity point in
  the interior of the elements in Figure~\subref{fig:dualconst}; 
  the dashed lines (gray boxes) are the new control volumes $V_i$ of $\TT^*$ and are 
  associated with $a_i\in\NN$.
  In~Figure~\subref{fig:edgeupwind} we see an example
  intersection $\tau_{17}=V_1\cap V_7\neq\emptyset$ of two neighboring cells
  $V_1,V_7 \in \TT^*$, where $\tau_{17}$ is the union of 
  two straight segments. For $a_3, a_4\in \NN$, where both $a_3$ and $a_4$ 
  lie on $\Gamma$, $\tau_{34}=V_3\cap V_4\neq\emptyset$ 
  is only a single segment.
  \label{fig:primaldual}}
\end{figure}
\subsection*{Dual mesh}
We construct the dual mesh $\TT^*$ from the primal mesh $\TT$ as follows.
In two dimensions we connect the center of gravity of an element $T\in\TT$ 
with the midpoint of the edges $E\in\EEt$; see 
Figure~\ref{fig:primaldual}\subref{fig:dualconst}, where 
the dashed lines are the new boxes, called control volumes. 
In three dimensions we connect the center of gravity of an element
$T\in\TT$ with the centers of gravity of the four faces $E\in\EEt$.
Furthermore, each center of gravity of a face $E\in\EEt$ is connected by straight lines
to the midpoints of its edges. The elements of this dual mesh $\TT^*$ are taken to be closed. Note that they are non-degenerate domains because of the non-degeneracy of the elements of the primal mesh.
Given a vertex $a_i\in\NN$ from the primal mesh $\TT$ ($i=1\ldots \#\NN$), there exists a unique box containing $a_i$. We thus number the elements of the dual mesh $V_i\in \TT^*$, following the numbering of vertices.
\begin{remark}
	In two dimensions, instead of starting the construction of the boxes in the center of gravity, we can use the center
	of the circle circumscribed to the element. Connecting these 
	points with the midpoints of the edges we form the so called Voronoi or perpendicular
	bisector meshes, since the connection between to neighbor's circumscribed circle points
	is perpendicular to the shared edge. 
	Our analysis works with such meshes as well.
\end{remark}
%
%
\subsection*{Discrete function spaces}
We define with
$\SS^1(\TT):=\set{v\in\CC(\Omega)}{v|_T \text{ affine for all } T\in\TT}$
the piecewise affine and globally continuous function space 
on $\TT$.
The space $\PP^0(\EEr)$ is the $\EEr$-piecewise constant function space. 
On the dual mesh $\TT^*$ we provide
$\PP^0(\TT^*):=\set{v\in L^2(\Omega)}{v|_V \text{ constant } V\in\TT^*}$.
With the aid of the characteristic function $\chi_i^*$ over the volume $V_i$ 
we write for $v_h^*\in\PP^0(\TT^*)$
\begin{align*}
  v_h^*=\sum_{x_i\in\NN}v_i^*\chi_i^*,
\end{align*}
with real coefficients $v_i$.
Furthermore, we define the $\TT^*$-piecewise constant interpolation operator
\begin{align}
  \label{eq:intoppiecewise}
  \IIh^*:\CC(\overline\Omega)\to\PP^0(\TT^*),\quad
  \IIh^*v:=\sum_{a_i\in\NN}v(a_i)\chi_i^*(x).
\end{align}
Because of the construction of the dual mesh from the primal mesh and the definition of $\IIh^*$
there hold the well known results:
\begin{lemma}
  \label{lem:proplindual}
  Let $T\in\TT$ and $E\in\EEt$. 
  For $v_h\in\SS^1(\TT)$ there holds
  \begin{align}
    \label{eq2:proplindual}
    &\int_E(v_h-\IIh^*v_h)\,ds=0,\\
    \label{eq3:proplindual}
    &\norm{v_h-\IIh^*v_h}{L^2(T)}\leq h_T \norm{\nabla v_h}{L^2(T)},\\
    \label{eq4:proplindual}
    &\norm{v_h-\IIh^*v_h}{L^2(E)}\leq C h_E^{1/2} \norm{\nabla v_h}{L^2(T)},
  \end{align}
  where the constants $C>0$ depend only on the shape regularity constant.
\end{lemma}
\begin{proof}
The proofs are standard. Note that for~\eqref{eq2:proplindual} we need the fact, 
that the dual mesh $\TT^*$ is constructed through the midpoint of an edge
$E\in\EE$ in the two dimensional case and the center of gravity point if
$E$ is a face in the three dimensional case.
A proof of~\eqref{eq3:proplindual} can be found in~\cite{Erath:2010-phd},
and~\eqref{eq4:proplindual} follows from~\eqref{eq3:proplindual} through
the standard trace inequality. Note that the above statements are
independent of the choice of the interior point in $T\in\TT$ 
for the $\TT^*$ construction.
\end{proof}
\subsection{The discrete system}
A classical finite volume discretization describes numerically a 
conservation law of the model problem, i.e., a quantity
in a volume can only change due to the inflow and outflow flux balance
through its boundary. More precisely, for our model problem we 
integrate~\eqref{eq1:model} over each dual control volume
$V\in\TT^*$ and apply the divergence theorem. If we use the transmission 
condition~\eqref{eq5:model}--\eqref{eq6:model} we thus get a 
balance equation for the interior
problem
\begin{align}
	\label{eq:FVMV}
	\begin{split}
	\int_{\partial V\backslash\Gamma}(-\A\nabla u_h&+\b u_h)\cdot\normal\,ds
	+\int_V\r u_h\,dx \\
	&+\int_{\partial V\cap \Gamma^{out}}\b\cdot\normal\, u_h\,ds
	-\int_{\partial V\cap\Gamma}\phi_h\,ds=\int_V f\,dx+\int_{\partial V\cap \Gamma}t_0\,ds
	\end{split}
\end{align}
for all $V\in\TT^*$. 
Note that the discretization in the interior domain follows 
along the dual mesh $\TT^*$.
Here, $u_h\in\SS^1(\TT)$ and $\phi_h\in\PP^0(\EEr)$
approximate $u$ and $\phi$, respectively. 
We can rewrite~\eqref{eq:FVMV} in terms of
a variational formulation;
\begin{align*}
\AA_V(u_h,v_h)-\spe{\phi_h}{\IIh^*v_h}{\Gamma}&=
\spl{f}{\IIh^*v_h}{\Omega}+\spe{t_0}{\IIh^*v_h}{\Gamma}	
\end{align*}
with the finite volume bilinear form $\AA_V:\SS^1(\TT)\times \SS^1(\TT)\to \mathbb R$ given by
\begin{align}
  \label{eq:fvmbilinear}
  \begin{split}
  \AA_V(u_h,v_h)&:=\sum_{a_i\in\NN}v_h(a_i)\bigg(
  \int_{\partial V_i\backslash\Gamma}(-\A \nabla u_h+\b u_h ) \cdot \normal\,ds\\
  &\qquad\qquad\quad+\int_{V_i}\r u_h\,dx 
  +\int_{\partial V_i\cap\Gamma^{out}}
	\b\cdot\normal\,u_h\,ds\bigg).
	\end{split}
\end{align}
 
\begin{remark}
Note that the trial and test spaces are different in practice.
The test functions in the finite volume part are in $\PP^0(\TT^*)$, which is
realized by taking nodal values $v_h(a_i)$ in~\eqref{eq:fvmbilinear} and by interpolation 
$\IIh^*v_h\in\PP^0(\TT^*)$ for $v_h\in\SS^1(\TT)$.
We have chosen the above definition to simplify the notation below.
\end{remark}
%
%
%
To complete the coupling formulation we choose as in the classical 
non-symmetric FEM-BEM formulation the BEM equation~\eqref{eq:cal1}
and replace the continuous ansatz and test spaces by 
discrete subspaces.
Finally, the discrete system reads:
find $u_h\in\SS^1(\TT)$ and $\phi_h\in\PP^0(\EEr)$ such that
\begin{subequations}
\begin{align}
\label{eq1:fvm}
  \AA_V(u_h,v_h)-\spe{\phi_h}{\IIh^*v_h}{\Gamma}&=\spl{f}{\IIh^*v_h}{\Omega}
  +\spe{t_0}{\IIh^*v_h}{\Gamma},\\
  \label{eq2:fvm}   
  \spe{\psi_h}{(1/2-\KK)u_h}{\Gamma}+\spe{\psi_h}{\VV\phi_h}{\Gamma}
  &=\spe{\psi_h}{(1/2-\KK)u_0}{\Gamma}	
\end{align}
\end{subequations}
for all $v_h\in\SS^1(\TT)$, $\psi_h\in\PP^0(\EEr)$.

As in the continuous case we write the system in a more compact way. We consider the product space $\HH_h:=\SS^1(\TT)\times \PP^0(\EEr)$, the bilinear form $\BB_V:\HH_h\times \HH_h \to \mathbb R$
\begin{align*}
	\BB_V((w_h,\phi_h);(v_h,\psi_h))
	:=&\AA_V(w_h,v_h)-\spe{\phi_h}{\IIh^*v_h}{\Gamma} \\
          &+\spe{\psi_h}{(1/2-\KK)w_h}{\Gamma}+\spe{\psi_h}{\VV\phi_h}{\Gamma},
\end{align*}
and the linear functional $F_V:\HH_h \to \mathbb R$
\begin{align}
	\label{eq:Ffvm}
	F_V((v_h,\psi_h)):=\spl{f}{\IIh^* v_h}{\Omega}+\spe{t_0}{\IIh^*v_h}{\Gamma}+\spe{\psi_h}{(1/2-\KK)u_0}{\Gamma}.  
\end{align}
The~\eqref{eq1:fvm}-\eqref{eq2:fvm} is equivalent to: find $\u_h\in\HH_h$ such that
\begin{align}
   \label{eq:FVMBEM}
   \BB_V(\u_h;\v_h)=F_V(\v_h) \qquad\text{for all } \v_h\in \HH_h.
\end{align}
%
\subsection{Upwind scheme}
\label{subsec:upwind}

In general it is a non trivial task to get a stable discrete solution
for convection dominated problems.
Finite volume schemes, however, allow an
easy upwind stabilization~\cite{Roos:1996-book}. 
If we want to apply an upwind scheme for the finite volume scheme,
we replace $\b u_h$ 
on the interior dual edges/faces $V_i\backslash \Gamma$
in $\AA_V$~\eqref{eq:fvmbilinear} by an upwinded approximation. Given $V_i\in \TT^*$, we consider the intersections with the neighboring cells 
$\tau_{ij}=V_i\cap V_j\neq\emptyset$ for $V_j \in \TT^*$. 
Note that in two dimensions $\tau_{ij}$ is the union of two straight segments or 
(when the associated vertices $a_i, a_j\in \NN$ lie on $\Gamma$) a single segment; see Figure~\ref{fig:primaldual}\subref{fig:edgeupwind}. 
In three dimensions $\tau_{ij}$ consists of one or two polygonal surfaces.
We then compute the averages
\begin{align*}
\beta_{ij}:=\frac1{|\tau_{ij}|}\int_{\tau_{ij}}\b\cdot\normal_i\,ds,
\qquad
\A_{ij} :=\frac1{|\tau_{ij}|} \int_{\tau_{ij}} \A \,ds,
\end{align*}
where $\normal_i$ points outward with respect to $V_i$, and
the parameter
\begin{align*}
\lambda_{ij}:=\Phi( \beta_{ij}|\tau_{ij}|/ \| \A_{ij}\|_\infty),
\end{align*}
for a weight function $\Phi:\mathbb R\to [0,1]$, which is being applied to the P\'eclet number. Then we consider the value
\begin{align*}
u_{h,ij}:=\lambda_{ij}u_h(a_i)+(1-\lambda_{ij})u_h(a_j)
\end{align*}
instead of $u_h$ when restricted to $\tau_{ij}\subset \partial V_i\setminus\Gamma$.  
In this work we choose 
the upwind value defined by the classical (full) upwind scheme by
\begin{align}
  \label{eq:weightfunctionfull}
\Phi(t):=(\operatorname{sign}(t)+1)/2,
\end{align}
i.e. $\lambda_{ij}=1$ for
$\beta_{ij}\geq 0$ and $\lambda_{ij}=0$ otherwise.
A second choice will be
\begin{align}
  \label{eq:weightfunction}
  \Phi(t):=\begin{cases}
   \displaystyle
    \min\big\{2|t|^{-1},1\big\}/2 \quad &\text{for } t<0,\\[1mm]
   \displaystyle 1-\min\big\{2|t|^{-1},1\big\}/2 \quad &\text{for } t\geq 0,
  \end{cases}
\end{align}
where we can steer the amount of upwinding
to reduce the excessive numerical diffusion.
Whenever we apply an upwind scheme for the convection part, we
replace the finite volume bilinear form $\AA_V$ in the
system~\eqref{eq1:fvm}--\eqref{eq2:fvm} by
\begin{align}
  \label{eq:fvmbilinearup}
  \begin{split}
  \AA_V^{up}(u_h,v_h)&:=\sum_{a_i\in\NN}v_h(a_i)\bigg(
  \int_{\partial V_i\backslash\Gamma}-\A \nabla u_h \cdot \normal\,ds
  +\int_{V_i}\r u_h\,dx \\
  &\qquad\qquad\quad
  +\sum_{j\in\NN_i}\int_{\tau_{ij}}\b \cdot\normal \,u_{h,ij}\,ds
  +\int_{\partial V_i\cap\Gamma^{out}}
	\b\cdot\normal\,u_h\,ds\bigg),
	\end{split}
\end{align}
where $\NN_i$ denotes the index set of nodes in $\TT$ 
of all neighbors of $a_i \in\NN$.
\section{Stability and convergence}\label{sec:stab}
In this section we want to introduce a stabilized FVM-BEM coupling
version of~\eqref{eq:FVMBEM} for analysis purposes only. 
As in~\eqref{eq:FEMBEMstab} 
we use the ``implicit theoretical'' stabilization of~\cite{wien}.

Similar as above we define 
 $\widetilde{\BB}_V:\HH_h\times\HH_h\to\R$ and $\widetilde F_V:\HH_h\to \mathbb R$ by
\begin{align}
\label{eq:Bstabfvm}
  \BBVstab{\u_h}{\v_h} &:=    
  	\BB_V(\u_h;\v_h)
	+ \beta P(\u_h) P(\v_h),\\
	\widetilde F_V(\v_h) &:= F_V(\v_h)+\beta\spe{1}{(1/2-\KK) u_0}{\Gamma}
		P(\v_h).
\end{align}
Then the stabilized FVM-BEM coupling reads:
find $\u_h\in\HH_h$, such that
\begin{align}
	\label{eq:FVMBEMstab}
	\BBVstab{\u_h}{\v_h}=\widetilde F_V(\v_h)
	\qquad \text{for all } \v_h\in \HH_h.
\end{align}
%
%
%
\begin{remark}
\label{rem:FEMBEM}
The discretized version of the stabilized FEM-BEM coupling reads with the 
stabilized weak form~\eqref{eq:FEMBEMstab}: find $\u_{h,FEM}\in\HH_h$ such that
\begin{align*}
\BBstab{\u_{h,FEM}}{\v_h}=\widetilde F(\v_h)
\qquad\text{for all } \v_h\in \HH_h.
\end{align*}
See also Remark~\ref{rem:discellipticity}.
\end{remark}
In the spirit of Lemma~\ref{lem:equi:fem} and~\cite[Theorem~14]{wien} we can state the
equivalence of the two presented FVM-BEM formulations.
\begin{lemma}\label{lem:equi:fvm}
The FVM-BEM coupling~\eqref{eq:FVMBEM}	
and its stabilization~\eqref{eq:FVMBEMstab}
are equivalent. 
The statement is also true if we replace 
$\AA_V $ by $\AA_V^{up} $ in the corresponding
bilinear forms.
\end{lemma}
\begin{proof}
In case of $\beta= 0$ the two formulations are obviously the same. Thus we only have to
consider $\beta= 1$. If $\u_h=(u_h,\phi_h)$ is a solution of~\eqref{eq:FVMBEM}, testing with $\v_h=(0,1)$ it follows that
\begin{align}\label{Pu=}
P(\u_h)=  \spe{1}{(1/2-\KK) u_h + \VV \phi_h}{\Gamma} = \spe{1}{(1/2-\KK) u_0}{\Gamma},
\end{align}
which means that we can add the stabilization term to~\eqref{eq:FVMBEM} to get the
stabilized version~\eqref{eq:FVMBEMstab}. Reciprocally, testing~\eqref{eq:FVMBEMstab} with $\v_h=(0,1)$, it follows that
\begin{align*}
  P(\u_h)
  (1+\spe{1}{\VV 1}{\Gamma}) = \spe{1}{(1/2-\KK) u_0}{\Gamma}
  (1+\spe{1}{\VV 1}{\Gamma}).
\end{align*}
Since the single layer operator is coercive,~\eqref{Pu=} follows and we can eliminate the $\beta$-dependent term in~\eqref{eq:FVMBEMstab} to get~\eqref{eq:FVMBEM}. Note that the proof is independent of the particular 
choice of the finite volume bilinear form, and it therefore holds for $\AA_V^{up}$ as well.
\end{proof}
The idea of our analysis is to estimate the difference of
the stabilized FEM-BEM coupling and
the stabilized FVM-BEM coupling. For that we need
the following two estimates, which are standard in the context of 
FVM~\cite{Ewing:2002-1,Chatzipantelidis:2002-1} with the
above constructed dual mesh, but here extended to the coupling problem.
\begin{lemma}
  \label{lem:fcompare}
For the difference of the right-hand side of~\eqref{eq:FEMBEMstab}
  and~\eqref{eq:FVMBEMstab}, there holds
  \begin{align}
   \label{eq:fcompare}
   \begin{split}
   |F(\v_h)-F_V(\v_h)|
   &\leq C \Big(\sum_{T\in\TT}h_T\norm{f}{L^2(T)}\norm{\nabla v_h}{L^2(T)}\\
    &\qquad +\sum_{E\in\EEr}h_E^{1/2}\norm{t_0-\overline{t}_0}{L^2(E)}
			   \norm{\nabla v_h}{L^2(T_E)}\Big)
   \end{split}
  \end{align} 
  for all $\v_h=(v_h,\psi_h)\in\HH_h$ with a constant $C>0$, 
  which depends only on the shape regularity constant.
  Here, $\overline{t}_0$ is the $\EEr$-piecewise integral mean of $t_0$ and $T_E$
  is the element associated with $E$.
\end{lemma}
\begin{proof}
	It is easy to see that from~\eqref{eq:Ffem} and~\eqref{eq:Ffvm} we get
\begin{align*}
		|F(\v_h)-F_V(\v_h)|=\spl{f}{v_h-\IIh^* v_h}{\Omega}
		+\spe{t_0}{v_h-\IIh^* v_h}{\Gamma}.
\end{align*}
The Cauchy-Schwarz inequality and~\eqref{eq2:proplindual}--\eqref{eq4:proplindual}
lead to the assertion.
\end{proof}
The next lemma gives us an estimate between the weak and
the finite volume bilinear
form for a function $v_h\in\SS^1(\TT)$.
\begin{lemma}
 \label{lem:bilinearcompare}
 Let us assume that
 $\b\cdot\normal$ is piecewise constant on $\Gamma^{in}$, 
 i.e. $\b\cdot\normal|_{\Gamma^{in}}\in\PP^0(\EEr^{in})$.
 For all $v_h,w_h\in\SS^1(\TT)$ there hold 
 \begin{align}
  \label{eq:bilinearcompare}
  |\AA(w_h,v_h)&-\AA_V(w_h,v_h)|
    \leq C_1\sum_{T\in\TT}\Big(h_T\norm{w_h}{H^1(T)}
	\norm{v_h}{H^1(T)}\Big),\\
  \label{eq:bilinearcompareup}
  |\AA(w_h,v_h)&-\AA_V^{up}(w_h,v_h)|
    \leq C_2\sum_{T\in\TT}\Big(h_T\norm{w_h}{H^1(T)}
	\norm{v_h}{H^1(T)}\Big),
 \end{align} 
  with constants $C_1,C_2>0$, depending only on the model data $\A$, $\b$, $\r$, and on the shape regularity constant.
\end{lemma}
\begin{proof}
  Let us define $v_h^*:=\II_h^*v_h\in\PP^0(\TT^*)$.
  Using integration by parts for
  $\AA(w_h,v_h)$ and $\AA_V(w_h,v_h)$ the lines in the proof 
  of~\cite[Lemma 5.2]{Erath:2012-1} show 
  with~\eqref{eq2:proplindual}
  \begin{align}
	\label{eq:bilinearcomparehelp}
   \begin{split}
   \AA(w_h,v_h)&-\AA_V(w_h,v_h)\\
   &=\sum_{T\in\TT}\Big(\spl{-(\div\A)\nabla w_h+(\div\b)w_h+
   \b\cdot \nabla w_h+\r w_h}{v_h-v_h^*}{T}\\
   &\qquad\qquad+\sum_{E\in\EEt}\spl{(\A-\overline{\A})\nabla w_h\cdot\normal}
   {v_h-v_h^*}{E}\\
   &\qquad\qquad-\sum_{E\in\EEt\cap\Gamma^{in}}
   \spl{\b\cdot\normal(w_h-\overline{w}_h)}{v_h-v_h^*}{E}\Big).
   \end{split}
  \end{align}
  Here, $\div\A$ is the divergence operator  
  applied to the columns of $\A$, 
  $\overline{\A}|_E\in \R^{d\times d}$ 
  is the  average of $\A$ over $E$, 
  and $\overline{w}_{h}\in\PP^0(\EEr)$ is the best $L^2(\Gamma)$ approximation of $w_h$.
With a standard approximation argument we prove
 \begin{align*}
    |\AA(w_h,v_h)&-\AA_V(w_h,v_h)|\\
     &\leq\sum_{T\in\TT}C_T\Big(\norm{w_h}{H^1(T)}\norm{v_h-v_h^*}{L^2(T)}
     +\sum_{E\in\EEt}h_E\norm{\nabla w_h}{L^2(E)}\norm{v_h-v_h^*}{L^2(E)}\\
     &\qquad\qquad+\sum_{E\in\EEt\cap\EEr^{in}}\norm{w_h-\overline{w}_h}{L^2(E)}
     \norm{v_h-v_h^*}{L^2(E)}\Big)
  \end{align*}
  with a constant $C_T>0$, which depends only on the shape regularity constant,
  and the model data $\A$, $\b$, and $\r$.
  The standard scaling inequalities
  $\norm{\nabla w_h}{L^2(E)}\leq C h_E^{-1/2}\norm{\nabla w_h}{L^2(T)}$ and
  $\norm{w_h-\overline{w}_h}{L^2(E)}\leq C h_T^{1/2}\norm{\nabla w_h}{L^2(T)}$, together with 
  and~\eqref{eq3:proplindual}-\eqref{eq4:proplindual},
prove~\eqref{eq:bilinearcompare}.
  To prove~\eqref{eq:bilinearcompareup} we write
  \begin{align*}
    |\AA(w_h,v_h)-\AA_V^{up}(w_h,v_h)|\leq
    |\AA(w_h,v_h)-\AA_V(w_h,v_h)|+|\AA_V(w_h,v_h)-\AA_V^{up}(w_h,v_h)|.
  \end{align*}
  Note that we can directly apply~\eqref{eq:bilinearcompare} for the first
  and~\cite[Lemma 6.1]{Erath:2012-1} for the second difference to
  show~\eqref{eq:bilinearcompareup}.
\end{proof}
\begin{remark}
  If $\A$ is $\TT$-piecewise constant, all parts with $\A$
  vanish in~\eqref{eq:bilinearcomparehelp} because of $\div\A=0$
  and~\eqref{eq2:proplindual} since $\nabla w_h$ is constant.
  This is well-known and if $\b=(0,0)^T$ and $\r=0$
  there even holds $\AA(w_h,v_h)=\AA_V(w_h,v_h^*)$, 
  see e.g.~\cite{Bank:1987-1,Hackbusch:1989-1}.
  Thus the following analysis also holds if $\A$ is $\TT$-piecewise constant.	
\end{remark}
Collecting all the results together we prove:
\begin{lemma}
 \label{lem:stabcomp}
 Let us assume that
 $\b\cdot\normal$ is piecewise constant on $\Gamma^{in}$, 
 i.e. $\b\cdot\normal|_{\Gamma^{in}}\in\PP^0(\EEr^{in})$.
 For all $\w_h=(w_h,\xi_h)\in\HH_h$ and $\v_h=(v_h,\psi_h)\in\HH_h$ there holds 
 \begin{equation}
  \label{eq:stabcomp}
  |\BBstab{\w_h}{\v_h}-
  \BBVstab{\w_h}{\v_h}|
    \leq C\sum_{T\in\TT}\Big(h_T\norm{w_h}{H^1(T)}
	\norm{v_h}{H^1(T)}\Big)
 \end{equation} 
  with a constant $C>0$, which depends only
 on the model data $\A$, $\b$, $\r$, and the shape regularity constant.
 The statement is also true if we replace 
 $\AA_V$ by $\AA_V^{up}$ in the corresponding
 bilinear forms.
\end{lemma}
\begin{proof}
We estimate
\begin{align*}
	|\BBstab{\w_h}{\v_h}-
	\BBVstab{\w_h}{\v_h}|
	&= |\AA(w_h,v_h)-\AA_V(w_h,v_h)
	-\spe{\xi_h}{v_h-\IIh^*v_h}{\Gamma}|\\
	&\leq C\sum_{T\in\TT}\Big(h_T\norm{w_h}{H^1(T)}
	\norm{v_h}{H^1(T)}\Big),
\end{align*}
where we used~\eqref{eq:bilinearcompare} and~\eqref{eq2:proplindual} 
since $\xi_h\in\PP^0(\EEr)$. Using ~\eqref{eq:bilinearcompareup}, 
the proof with $\AA_V^{up}$ follows from this bound.
\end{proof}
\begin{theorem}[Stability]
\label{th:stability}
There exists $H>0$ such that the following statement is valid provided
that $\TT$ is sufficiently fine, i.e., $h:=\max_{T\in\TT}h_T<H$:
Let $\lambda_{\min}(\A)>C_{\KK}/4$ with the contraction constant
$C_{\KK}\in[1/2,1)$  of the double layer potential.
Furthermore, let
 $\b\cdot\normal$ be piecewise constant on $\Gamma^{in}$, 
 i.e. $\b\cdot\normal|_{\Gamma^{in}}\in\PP^0(\EEr^{in})$.
Then, there exists a constant $C_{\rm Vstab}>0$ such that
\begin{align}
\label{eq:stabV}
\BBVstab{\v_h}{\v_h}\geq C_{\rm Vstab} \norm{\v_h}{\HH}^2
\qquad\text{for all } \v_h\in \HH_h.	
\end{align}
The constant $C_{\rm Vstab}>0$ depends only
 on the model data $\A$, $\b$, $\r$, the contraction constant $C_{\KK}$, 
 and the shape regularity constant.
 The statement also holds if we replace 
 $\AA_V$ by $\AA_V^{up}$ in the corresponding
 bilinear forms.
\end{theorem}
\begin{proof}
	From~\eqref{eq:stabcomp}  we see with $C'>0$
	\begin{align*}
  \BBVstab{\v_h}{\v_h}
    &\geq \BBstab{\v_h}{\v_h}
	- C'h\| v_h\|_{H^1(\Omega)}^2 .
	\end{align*}
	The stability estimate~\eqref{eq:ellipticity:konv} provides
	$\BBstab{\v_h}{\v_h}\geq C_{\rm stab}' \norm{\v_h}{\HH}^2$ 
	with $C_{\rm stab}'>0$, 
	which proves the coercivity estimate for $h$ small enough.
	The proof with $\AA_V^{up}$
	is the same.
\end{proof}
\begin{theorem}[A~priori convergence estimate]
\label{th:fvmapriori}
There exists $H>0$ such that the following statement is valid provided
that $\TT$ is sufficiently fine, i.e., $h:=\max_{T\in\TT}h_T<H$:
Let $\lambda_{\min}(\A)>C_{\KK}/4$ with the contraction constant
$C_{\KK}\in[1/2,1)$ of the double layer potential $\KK$.
Furthermore, let
$\b\cdot\normal$ be piecewise constant on $\Gamma^{in}$, 
i.e. $\b\cdot\normal|_{\Gamma^{in}}\in\PP^0(\EEr^{in})$.
For the solution $\u=(u,\phi)\in\HH= H^1(\Omega)\times H^{-1/2}(\Gamma)$ 
of our model problem~\eqref{eq:FEMBEMstab} and the discrete solution
$\u_h=(u_h,\phi_h)\in\HH_h= \SS^1(\TT)\times \PP^0(\EEr))$ 
of our FVM-BEM coupling~\eqref{eq:FVMBEMstab}
there holds
\begin{align*}
\norm{\u-\u_h}{\HH}\leq C_{\rm est}\Big(
h\norm{f}{L^2(\Omega)}
+h^{1/2}\norm{t_0-\overline{t}_0}{L^2(\Gamma)}
+(1+h)\inf_{\v_h\in\HH_h}\norm{\u-\v_h}{\HH}
+h\norm{\u}{\HH}\Big),
\end{align*}
where $\overline{t}_0$ is the $\EEr$-piecewise integral mean of $t_0$.
The constant $C_{\rm est}>0$ depends only
 on the model data $\A$, $\b$, $\r$, the contraction constant $C_{\KK}$, 
 and the shape regularity constant.
In particular, if $u\in H^2(\Omega)$, $\phi\in H^{1/2}(\EEr)$, and
$t_0\in H^{1/2}(\EEr)$, where
\begin{align*}
H^{1/2}(\EEr):=\set{v\in L^2(\Gamma)}{v|_E\in H^{1/2}(E)\text{ for all }E\in\EEr},
\end{align*}
we have first order convergence
\begin{align*}
\norm{\u-\u_h}{\HH}=\mathcal{O}(h).
\end{align*}
The statement is also true if we replace 
$\AA_V$ by $\AA_V^{up}$ in the corresponding
bilinear forms.
\end{theorem}
In the following proof of Theorem~\ref{th:fvmapriori}, 
we write the symbol $\lesssim$, if an
estimate holds up to a multiplicative constant, which depends only
on the model data $\A$, $\b$, $\r$, the contraction constant $C_{\KK}$, 
and the shape regularity constant.
\begin{proof} 
For arbitrary $\v_h=(v_h,\psi_h)\in\HH_h$ we
define $\w_h=(w_h,\varphi_h):=\u_h-\v_h\in\HH_h$. Then we 
get with~\eqref{eq:stabV} 
\begin{align*}
\norm{\u_h-\v_h}{\HH_h}^2
&\lesssim\BBVstab{\u_h}{\w_h}-\BBVstab{\v_h}{\w_h}\\
&=\widetilde F_V(\w_h)-\widetilde F(\w_h)+ \BBstab{\u}{\w_h}-\BBVstab{\v_h}{\w_h},
\end{align*}
where we used the finite volume discrete system~\eqref{eq:FVMBEMstab}
and the FEM-BEM bilinear from~\eqref{eq:FEMBEMstab} with discrete
test functions $\w_h\in\HH_h$.
Since $\widetilde F_V(\w_h)-\widetilde F(\w_h)=F_V(\w_h)-F(\w_h)$ 
we apply~\eqref{eq:fcompare}
and insert $\v_h$ to estimate
\begin{align*}
\norm{\u_h-\v_h}{\HH}^2&\lesssim 
h\norm{f}{L^2(\Omega)}\norm{\nabla w_h}{L^2(\Omega)}+
h_{\EEr}^{1/2}\norm{t_0-\overline{t}_0}{L^2(\Gamma)}\norm{\nabla w_h}{L^2(\Omega)}\\
&\quad+\BBstab{\u-\v_h}{\w_h} +\BBstab{\v_h}{\w_h}-\BBVstab{\v_h}{\w_h},
\end{align*}
where $h_{\EEr}:=\max_{E\in\EEr}h_E$.
For the second term on the right-hand side we apply the boundedness of
$\widetilde{\BB}$
and we estimate the last two terms with~\eqref{eq:stabcomp}.
Thus we obtain
\begin{align*}
\norm{\u_h-\v_h}{\HH}^2& \lesssim 
h\norm{f}{L^2(\Omega)}\norm{\nabla w_h}{L^2(\Omega)}+
h_{\EEr}^{1/2}\norm{t_0-\overline{t}_0}{L^2(\Gamma)}\norm{\nabla w_h}{L^2(\Omega)}\\
&\quad+\norm{\u-\v_h}{\HH}\norm{\w_h}{\HH}+h\norm{v_h}{H^1(\Omega)}
	\norm{w_h}{H^1(\Omega)}.
\end{align*}
Finally with 
$\norm{w_h}{H^1(\Omega)}\leq \norm{\w_h}{\HH}=\norm{{\u}_h-\v_h}{\HH}$
we get
\begin{align*}
\norm{\u_h-\v_h}{\HH}\lesssim h\norm{f}{L^2(\Omega)}
+h_{\EEr}^{1/2}\norm{t_0-\overline{t}_0}{L^2(\Gamma)}
+\norm{\u-\v_h}{\HH}+h\norm{v_h}{H^1(\Omega)}.
\end{align*}
With $\norm{v_h}{H^1(\Omega)}\leq \norm{\u-\v_h}{\HH}+\norm{\u}{\HH}$
and
\begin{align*}
\norm{\u-\u_h}{\HH}\leq\norm{\u-\v_h}{\HH}+\norm{\u_h-\v_h}{\HH}	
\end{align*}
we get the assertion with $h_{\EEr}\geq h$.
The proof with $\AA_V^{up}$
is the same.
\end{proof}

\begin{remark}
In~\cite[see Remark 5.1]{Erath:2012-1}, 
where we consider a FVM-BEM coupling with a three field
coupling approach, we have the constraint $\phi\in L^2(\Gamma)$ 
in the case $\gamma(x)=0$ from assumption~\eqref{eq:bcestimate1} to get convergence
and an error estimate. Note that this regularity is not needed therein to 
prove existence and uniqueness, see~\cite[see Remark 5.2]{Erath:2012-1}.
Furthermore, there is also an additional assumption 
necessary in the case $\gamma(x)=0$, namely $\div\b+\r=0$ in $\Omega$ and 
$\b\cdot\normal=0$ on $\Gamma^{in}$.
Thus Theorem~\ref{th:stability}, which essentially shows existence and
uniqueness of a discrete solution, and Theorem~\ref{th:fvmapriori}
for our non-symmetric FVM-BEM coupling
are much stronger than what is available for the three field
FVM-BEM coupling. However, the constraint $\lambda_{\min}(\A)>C_{\KK}/4$ 
on the eigenvalues of $\A$
is not needed for the three field FVM-BEM coupling.
\end{remark}

\section{Numerical results}
\label{sec:numerics}
In this section we verify our new coupling with three examples.
We stress that in all experiments we consider the discrete
FVM-BEM system~\eqref{eq1:fvm}--\eqref{eq2:fvm} and~\eqref{eq:FVMBEM},
respectively, where we replace $\AA_V$ defined in~\eqref{eq:fvmbilinear}
by the upwind form $\AA_V^{up}$ defined in~\eqref{eq:fvmbilinearup} 
if we use an upwind scheme 
for the convection part.
We mention
once again, that the equivalent 
stabilized FVM-BEM system~\eqref{eq:FVMBEMstab} is only needed for 
theoretical reasons.

All the numerical experiments are done in {\sc Matlab} on a standard
laptop with a dual core $2.8$ GHz processor and $16$ GB memory. 
Only the implementation of the matrices 
resulting from the $\VV$ and $\KK$
expressions is done in \emph{C} using the \emph{mex}-interface of 
{\sc Matlab}~\cite{Erath:2012-1,Erath:2013-2}. 
As introduced earlier, we use the equivalence of norms
$\norm{\phi-\phi_h}{H^{-1/2}(\Gamma)}^2\sim \norm{\phi-\phi_h}{\VV}^2:=
\spe{\VV(\phi-\phi_h)}{\phi-\phi_h}{\Gamma}$,
to calculate the conormal error $\phi-\phi_h$.
Then $\norm{\phi-\phi_h}{\VV}$ leads to an approximation 
of a double integral
by quadrature. The details can be found in~\cite{Erath:2010-phd,Erath:2012-1,Erath:2013-2}.
In all experiments and in each iteration, $\TT$ consists of triangles,
which are up to rotation congruent. In this work we only
consider uniform mesh refinement,
i.e., we divide all triangles 
by four triangles.

\subsection{Mexican hat problem}
\label{ex:bsp1}
\begin{figure}
\begin{center}
\begin{psfrags}%
\psfragscanon%
%
\psfrag{s08}[b][b]{\small error}%
\psfrag{s03}[t][t]{\small number of elements}%
\psfrag{s07}[l][l]{\scriptsize $1$}%
\psfrag{s04}[l][l]{\scriptsize $1/2$}%
\psfrag{s10}[l][l]{\scriptsize $1$}%
\psfrag{s09}[l][l]{\scriptsize $3/4$}%
\psfrag{s12}[l][l]{\scriptsize $1$}%
\psfrag{s11}[l][l]{\scriptsize $1$}%
\psfrag{s01}[l][l]{\small$\norm{\nabla(u-u_h)}{L^2(\Omega)}$}%
\psfrag{s02}[l][l]{\small$\norm{\phi-\phi_h}{\VV}$}%
\psfrag{s13}[l][l]{\small$\norm{u-u_h}{L^2(\Omega)}$}%
%
%
\psfrag{x01}[t][t]{\scriptsize$10^{{1}}$}%
\psfrag{x02}[t][t]{\scriptsize$10^{{2}}$}%
\psfrag{x03}[t][t]{\scriptsize$10^{{3}}$}%
\psfrag{x04}[t][t]{\scriptsize$10^{{4}}$}%
\psfrag{x05}[t][t]{\scriptsize$10^{{5}}$}%
\psfrag{x06}[t][t]{\scriptsize$10^{{6}}$}%
\psfrag{x07}[t][t]{\scriptsize$10^{{7}}$}%
%
\psfrag{v01}[r][r]{\scriptsize$10^{{-7}}$}%
\psfrag{v02}[r][r]{\scriptsize$10^{{-6}}$}%
\psfrag{v03}[r][r]{\scriptsize$10^{{-5}}$}%
\psfrag{v04}[r][r]{\scriptsize$10^{{-4}}$}%
\psfrag{v05}[r][r]{\scriptsize$10^{{-3}}$}%
\psfrag{v06}[r][r]{\scriptsize$10^{{-2}}$}%
\psfrag{v07}[r][r]{\scriptsize$10^{{-1}}$}%
\psfrag{v08}[r][r]{\scriptsize$10^{{0}}$}%
\psfrag{v09}[r][r]{\scriptsize$10^{{1}}$}%
%
\includegraphics[width=0.8\textwidth]{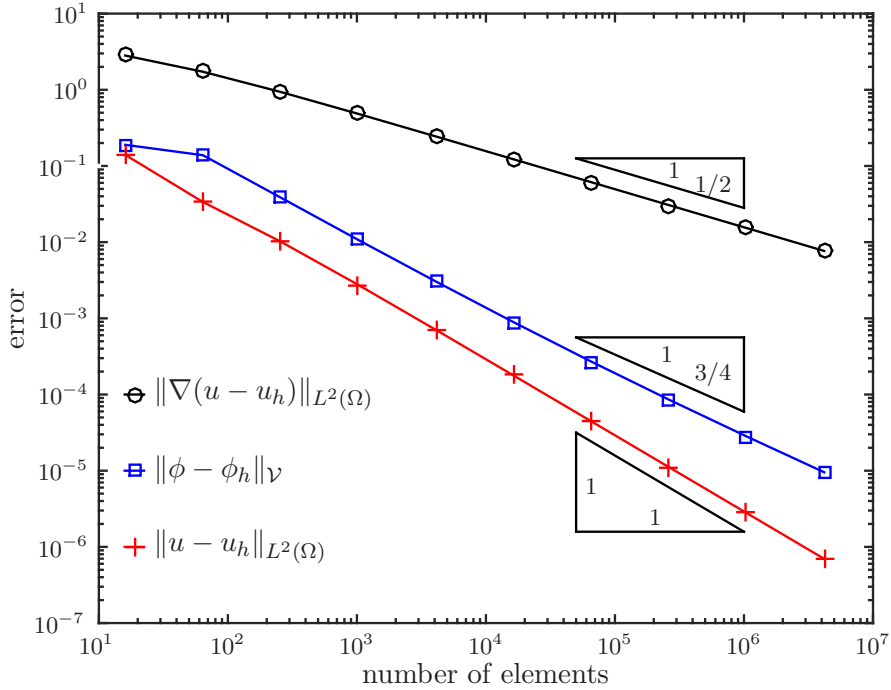}%
\end{psfrags}%
\end{center}
\caption{\label{fig:mexicanhaterror}
The error $\norm{\nabla(u-u_h)}{L^2(\Omega)}$ in the $H^1$-semi-norm,
the error $\norm{u-u_h}{L^2(\Omega)}$ in the $L^2$-norm, and the 
conormal error $\norm{\phi-\phi_h}{\VV}$ in the $\VV$-norm in the 
example in subsection~\ref{ex:bsp1} for uniform mesh-refinement.}
\end{figure}
\begin{figure}
\begin{center}
\includegraphics[width=0.8\textwidth]{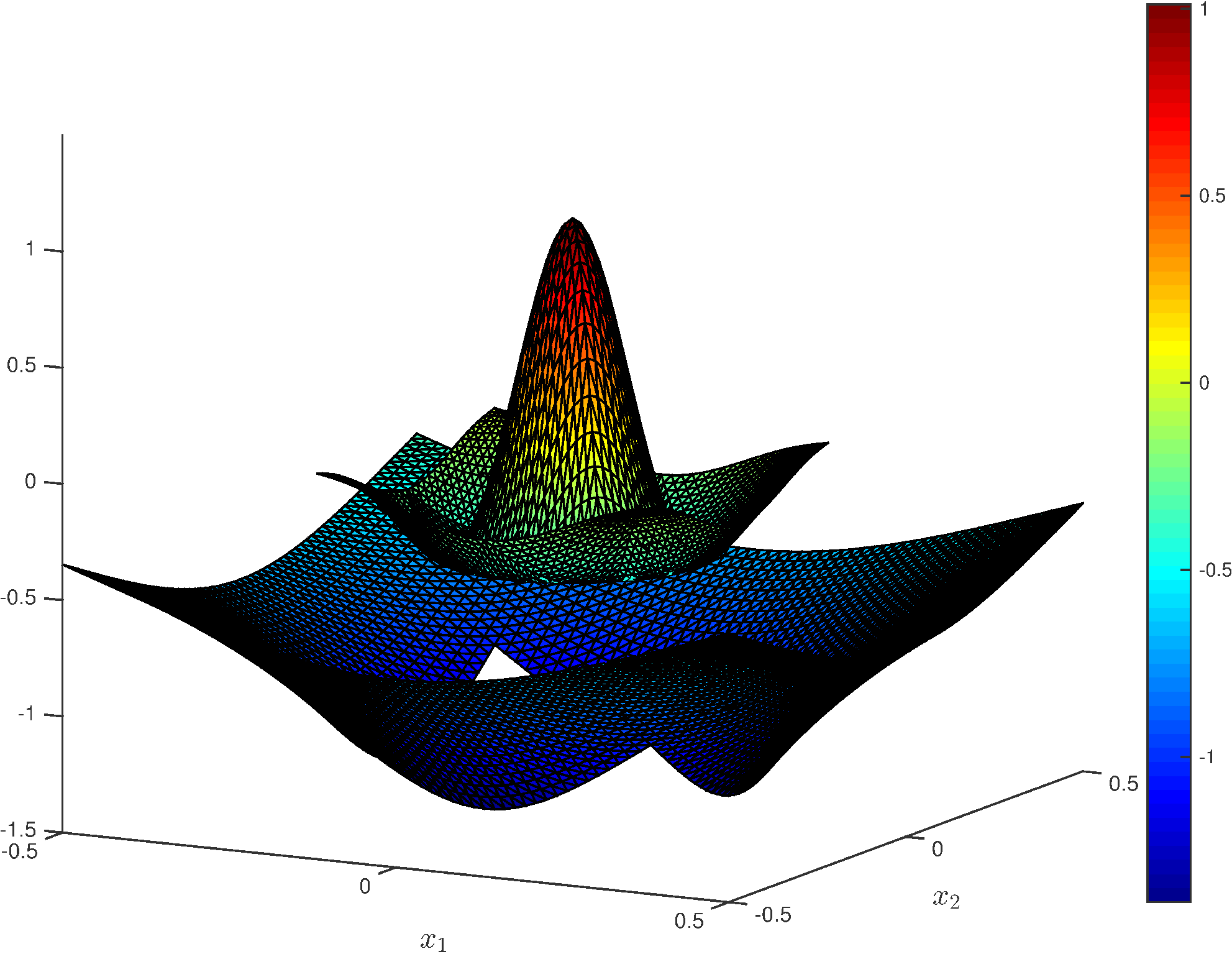}%
\end{center}
\caption{\label{fig:mexicanhat-sol}
Interior and exterior solution
on an uniformly generated mesh with $4096$ elements in the example in  
subsection~\ref{ex:bsp1}.}
\end{figure}
We consider the square $\Omega=(-1/4,1/4)^2$.  
We take the exact solution to be 
 $u(x_1,x_2) = (1-100x_1^2-100x_2^2)e^{-50(x_1^2+x_2^2)}$ in the interior domain and
$u_e(x_1,x_2)=\log\sqrt{x_1^2+x_2^2}$ in the exterior.
The diffusion matrix is
\begin{align}
	\label{eq:bsp1A}
	\A=  \left ( \begin{array}{rr}
  10+\cos x_1 & 160\,x_1 x_2 \\
  160\,x_1 x_2 & 10+\sin x_2
  \end{array}\right),
\end{align}
and we take $\b=(0,0)^T$ and $\r=0$.
Note that in $\Omega$ we have $\lambda_{\min}(\A)=0.342278$ and
$\lambda_{\max}(\A)=10.247271$.
The right-hand side $f$ and the jumps $u_0$ and $t_0$ are calculated appropriately.
We stress that $u$ and $u_e$ are smooth in $\Omega$ and $\Omega_e$, respectively.
Therefore, we expect a convergence order $\OO(h^{1})$ for a first order numerical 
scheme in the $H^1$-norm, where $h:=\max_{T\in\TT}h_T$ 
denotes the uniform mesh-size. This 
corresponds to order $\OO(N^{-1/2})$ with respect to the number of
elements $N$ of $\TT$. The initial mesh $\TT^{(0)}$ consists of $16$ triangles.
Figure~\ref{fig:mexicanhaterror} shows the curves of the interior error $u-u_h$ in 
the $H^1$-semi-norm and $L^2$-norm, respectively, and the 
conormal error of $\phi-\phi_h$ in the $\VV$-norm.
Both axes are scaled logarithmically; i.e.,
a straight line $g$ with slope $-p$ corresponds to a dependence 
$g=\OO(N^{-p})=\OO(h^{2p})$. 
The interior $H^1$-semi-norm error 
leads to a convergence 
order $\OO(N^{-1/2})$, whereas the corresponding $L^2$-norm error decreases
with $\OO(N^{-1})$. Thus, the error in $H^1$-norm behaves like $\OO(N^{-1/2})$.
The convergence of the BEM conormal quantity is optimal in the sense
of $\OO(N^{-3/4})$ due to the smooth solution.
Altogether we see $\norm{\u-\u_h}{\HH}=\OO(N^{-1/2})=\OO(h)$
with $\u=(u,\phi)$ and $\u_h=(u_h,\phi_h)\in\HH_h$,
which was shown in Theorem~\ref{th:fvmapriori} for smooth solutions.

Figure~\ref{fig:mexicanhat-sol} shows
the solution in $\Omega$ and parts of $\Omega_e$.
We observe the jump
on the coupling boundary $\Gamma$ and remark that the BEM solution is generated
pointwise with the aid of the exterior representation formula~\eqref{eq:repformula}
on a uniform grid.
For points on the boundary $\Gamma$ coming
from the exterior domain, we use the exterior trace of~\eqref{eq:repformula}.
Note that instead of~\eqref{eq:cal1} this approximated trace reads
\begin{align}
  \label{eq:repformulaangle}
  u_{e,h}|_\Gamma(x)=-(\VV\phi_h)(x)+
  \Big(\big(\KK+\frac{\varphi}{2\pi}\big)(u_h-u_0)\Big)(x)
\end{align}
for a point evaluation $x\in\Gamma$, where $\varphi$ is the interior angle of the
intersection of the two tangential vectors in $x$.
\begin{remark}
For this example $\gamma(x)=0$ from assumption~\eqref{eq:bcestimate1}.
Thus the analysis needs the stabilized bilinear form~\eqref{eq:Bstabfvm} with
$\beta=1$ from~\eqref{eq:beta}.
In particular, we have the condition
$\lambda_{\min}(\A)>C_{\KK}/4$, where $C_{\KK}\in[1/2,1)$ is the contraction constant 
of the double layer potential $\KK$. Note that our $\A$ with 
$\lambda_{\min}(\A)=0.342278$ fulfills this constraint.
If one replace both values of $160$ by $165$ we get 
$\lambda_{\min}(\A)=0.003033$ which contradicts the bound. 
However, the experiences (not plotted here) show
the right convergence behavior.
This confirms similar observations for FEM-BEM couplings, e.g.~\cite{wien}.
In particular, the bound seems to be a theoretical bound also for our FVM-BEM
coupling approach.
\end{remark}

\subsection{Convection-diffusion problem}
\label{ex:bsp2}
\begin{figure}
\begin{center}
\begin{psfrags}%
\psfragscanon%
%
\psfrag{s03}[l][l]{\scriptsize $1$}%
\psfrag{s04}[l][l]{  \scriptsize    $1/2$    }%
\psfrag{s05}[l][l]{\scriptsize $1$}%
\psfrag{s01}[l][l]{ \scriptsize    $3/4$  }%
\psfrag{s10}[l][l]{\scriptsize $1$}%
\psfrag{s11}[l][l]{ \scriptsize      $1$      }%
\psfrag{s06}[l][l]{\small$\norm{u-u_h}{L^2(\Omega)}$}%
\psfrag{s09}[l][l]{\small$\norm{\nabla(u-u_h)}{L^2(\Omega)}$}%
\psfrag{s08}[l][l]{\small$\norm{\phi-\phi_h}{\VV}$}%
\psfrag{s12}[b][b]{\small error}%
\psfrag{s13}[t][t]{\small number of elements}%
%
%
\psfrag{x01}[t][t]{\scriptsize$10^{{1}}$}%
\psfrag{x02}[t][t]{\scriptsize$10^{{2}}$}%
\psfrag{x03}[t][t]{\scriptsize$10^{{3}}$}%
\psfrag{x04}[t][t]{\scriptsize$10^{{4}}$}%
\psfrag{x05}[t][t]{\scriptsize$10^{{5}}$}%
\psfrag{x06}[t][t]{\scriptsize$10^{{6}}$}%
\psfrag{x07}[t][t]{\scriptsize$10^{{7}}$}%
%
\psfrag{v01}[r][r]{\scriptsize$10^{{-7}}$}%
\psfrag{v02}[r][r]{\scriptsize$10^{{-6}}$}%
\psfrag{v03}[r][r]{\scriptsize$10^{{-5}}$}%
\psfrag{v04}[r][r]{\scriptsize$10^{{-4}}$}%
\psfrag{v05}[r][r]{\scriptsize$10^{{-3}}$}%
\psfrag{v06}[r][r]{\scriptsize$10^{{-2}}$}%
\psfrag{v07}[r][r]{\scriptsize$10^{{-1}}$}%
\psfrag{v08}[r][r]{\scriptsize$10^{{0}}$}%
\psfrag{v09}[r][r]{\scriptsize$10^{{1}}$}%
%
\includegraphics[width=0.8\textwidth]{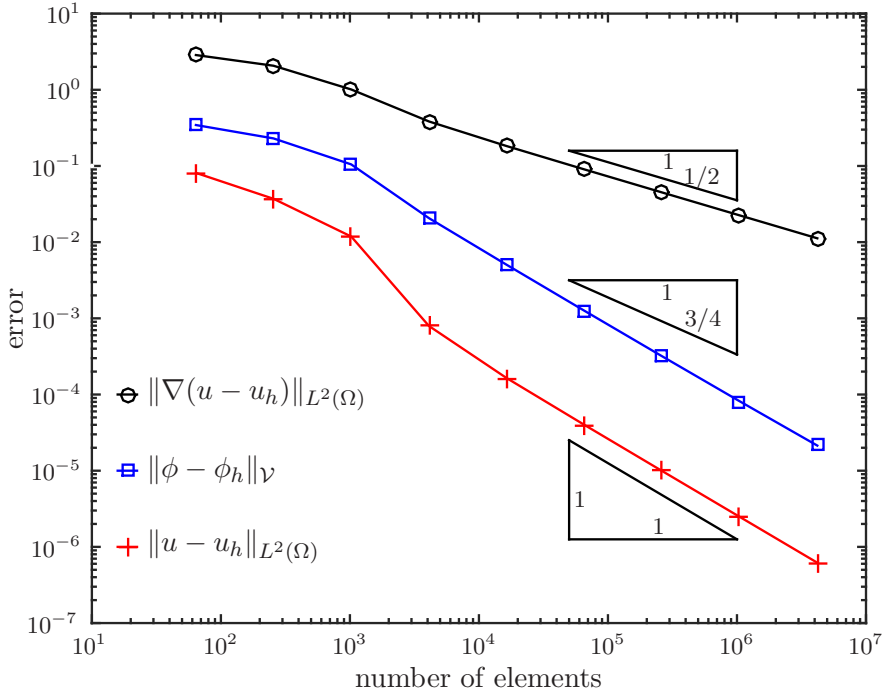}%
\end{psfrags}%
\end{center}
\caption{\label{fig:bsp2-error}
The error $\norm{\nabla(u-u_h)}{L^2(\Omega)}$ in the $H^1$ semi-norm,
the error $\norm{u-u_h}{L^2(\Omega)}$ in the $L^2$ norm, and the 
conormal error $\norm{\phi-\phi_h}{\VV}$ in the $\VV$ norm in the 
example in subsection~\ref{ex:bsp2} for uniform mesh-refinement.}
\end{figure}
\begin{figure}
\begin{center}
\includegraphics[width=0.8\textwidth]{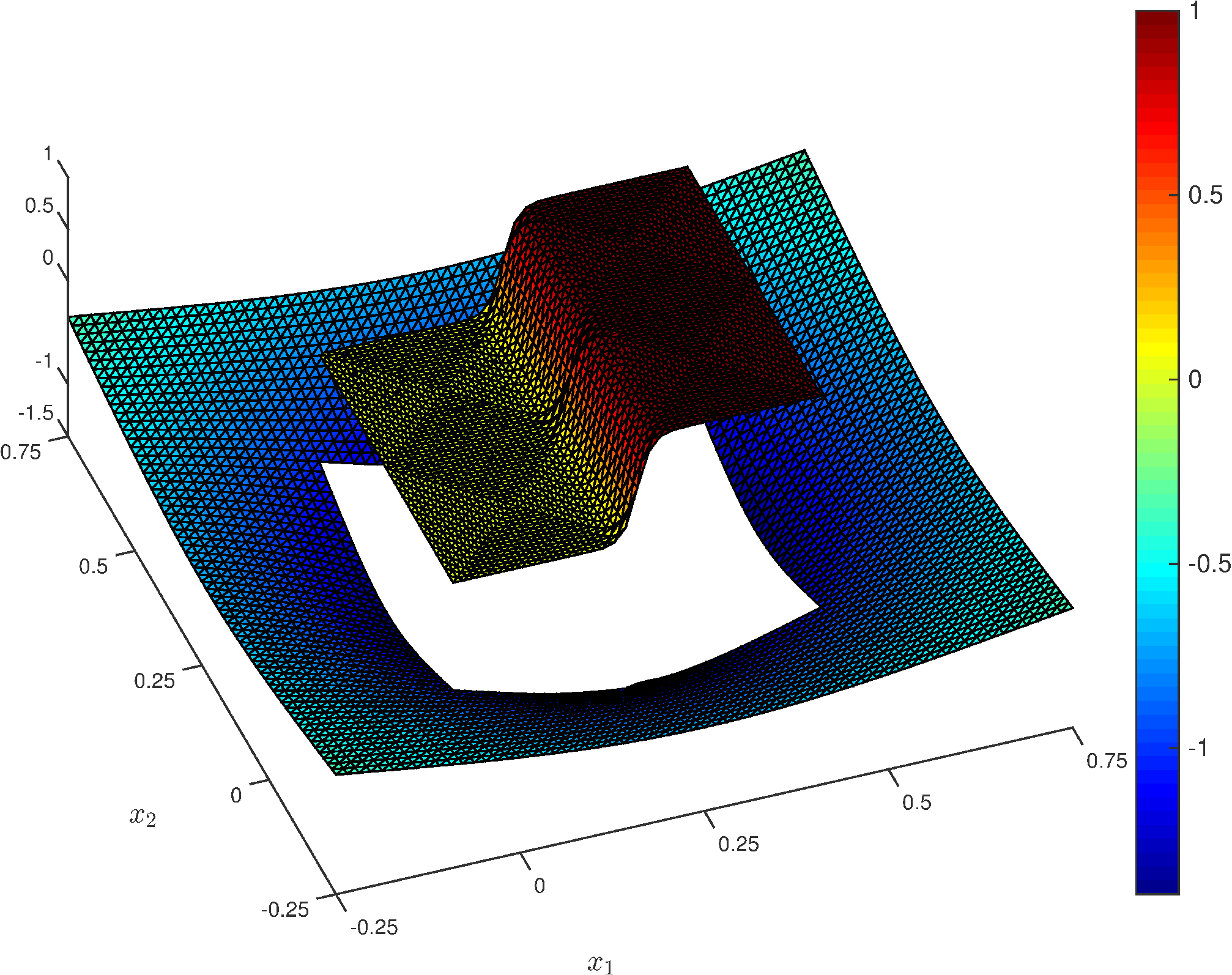}%
\end{center}
\caption{\label{fig:bsp2-sol}
Interior and exterior solution
with a weighted upwinding stabilization on
an uniformly generated mesh with $4096$ elements in the example in  
subsection~\ref{ex:bsp2}.}
\end{figure}
We consider the model problem on the square domain $\Omega=(0,1/2)\times(0,1/2)$.
We choose a fixed
diffusion matrix of $\A= 0.5\,\I$,
a convection field $\b=(1000x_1,0)^T$ and a reaction coefficient $\r=0$.
Note that for this problem we do not have an inflow boundary $\Gamma^{in}$
and thus~\eqref{eq5:model} is not needed.
For all calculations we use the upwind discrete coupling with the weighting
function $\Phi$ defined in~\eqref{eq:weightfunction}.
We prescribe an analytical solution
\begin{align*}
u(x_1,x_2)=0.5\bigg(1-\tanh\Big(\frac{0.25-x_1}{0.02}\Big)\bigg)
\end{align*} 
for the interior domain $\Omega$
and
\begin{align*}
u_e(x_1,x_2)=\log\sqrt{(x_1-0.25)^2+(x_2-0.25)^2}
\end{align*}
for the exterior domain $\Omega_e$.
We calculate the right-hand side $f$ and the jumps $u_0$ and $t_0$ appropriately. 
Note that $\lambda_{\min}(\A)=0.5$ and that the problem is highly convection dominated.

The initial mesh $\TT^{(0)}$ consists of $16$ triangles.
In Figure~\ref{fig:bsp2-error} we plot the convergence rate for uniform
mesh-refinement with respect to the number of elements in $\TT$.
Since the interior and exterior solution are smooth as in the previous example
in subsection~\ref{ex:bsp1}, we observe a similar convergence behavior,
in particular, $\norm{\u-\u_h}{\HH}=\OO(N^{-1/2})=\OO(h)$
with $\u=(u,\phi)$ and $\u_h=(u_h,\phi_h)\in\HH_h$, which
also confirms Theorem~\ref{th:fvmapriori}.
However, due to the strong convection, we have a preasymptotic phase.
We want to mention, that without any upwind stabilization, it is not possible
to get a stable solution even for more than $4$ million elements, 
which is the last mesh in our calculation.
In Figure~\ref{fig:bsp2-sol} we plot the interior and exterior solution.
To resolve the shock at $x_1=0.25$ better and thus to reduce the effects 
to the exterior
domain, one can use adaptive mesh refinement as in~\cite{Erath:2013-1}. 
However, this is beyond this work. 

\subsection{A more practical example}
\label{ex:bsp3}
%
\begin{figure}
\begin{center}
	\subfigure[\label{subfig:bsp3solnoupwind}Interior numerical solution without stabilization.]
	{\includegraphics[width=0.45\textwidth]
	{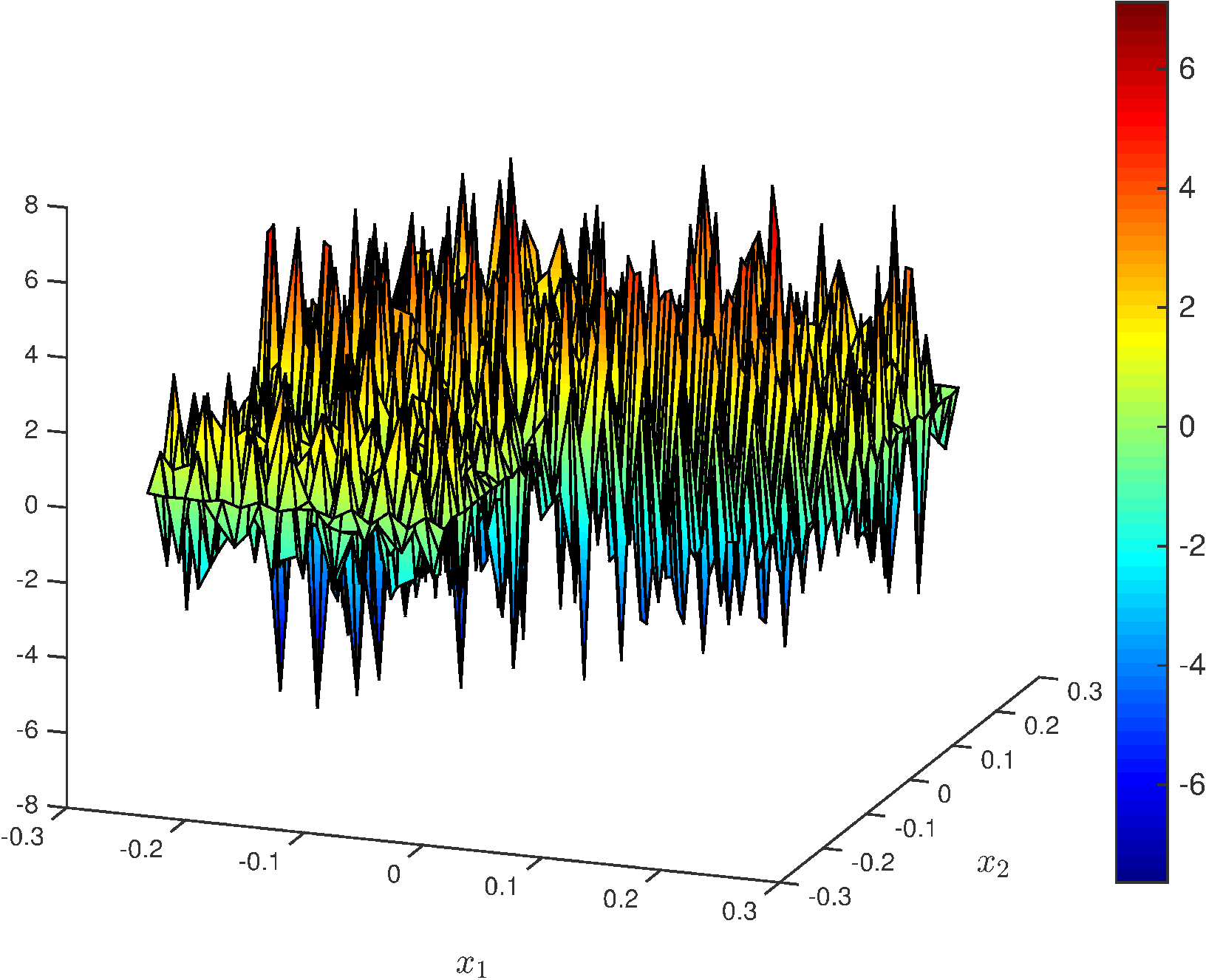}}
	\subfigure[\label{subfig:bsp3contour}Contour lines with full upwind.]
	{\includegraphics[width=0.5\textwidth]
	{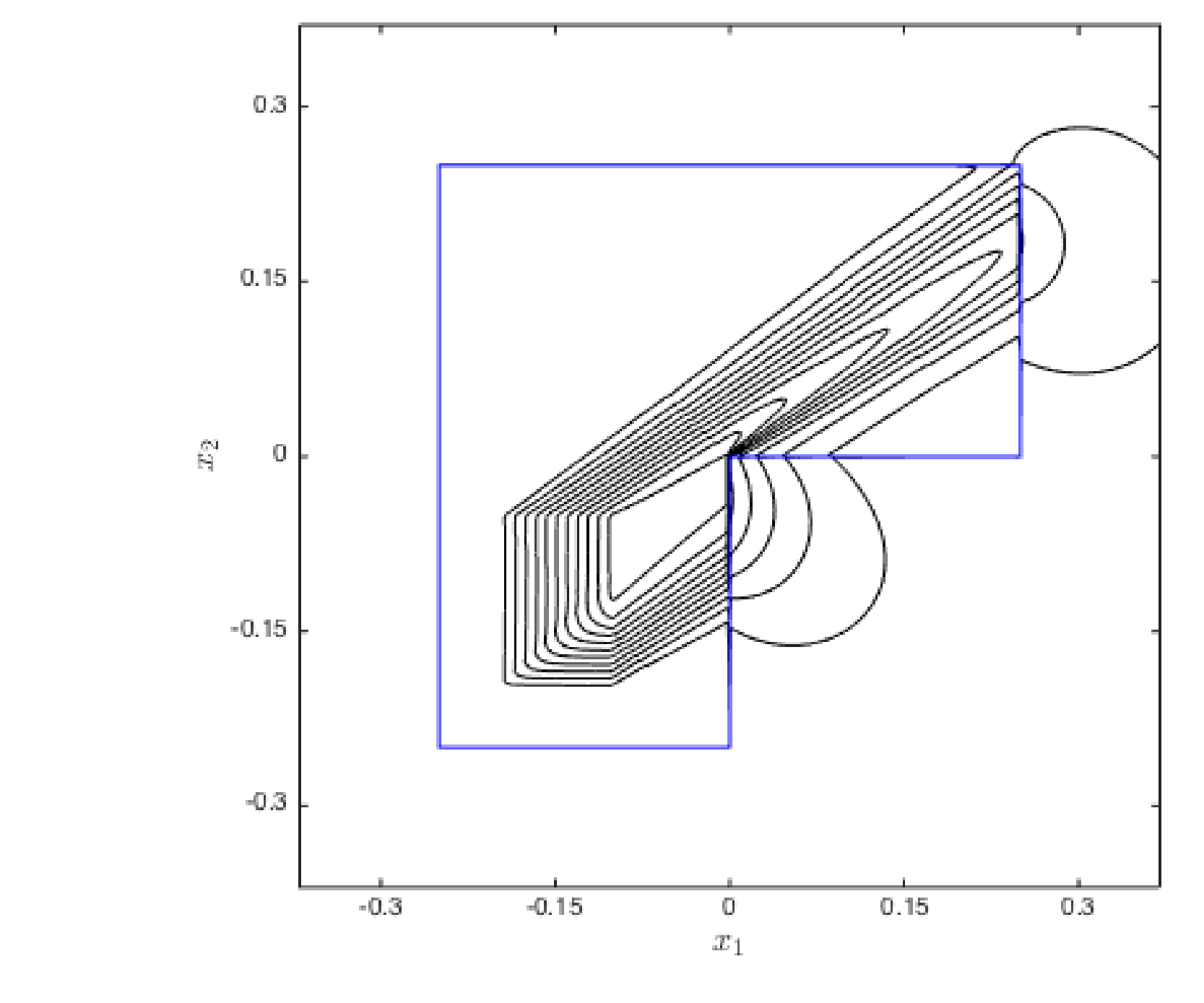}}
\end{center}
\caption{\label{fig:bsp3-solint}A convection approximation
without upwinding or any other stabilization leads to strong 
oscillations in~\subref{subfig:bsp3solnoupwind} in the example in  
subsection~\ref{ex:bsp3}. In~\subref{subfig:bsp3contour} we see
the transmission effects of the interior and exterior problem through
a contour line plot.}
\end{figure}

\begin{figure}
\begin{center}
\includegraphics[width=0.8\textwidth]{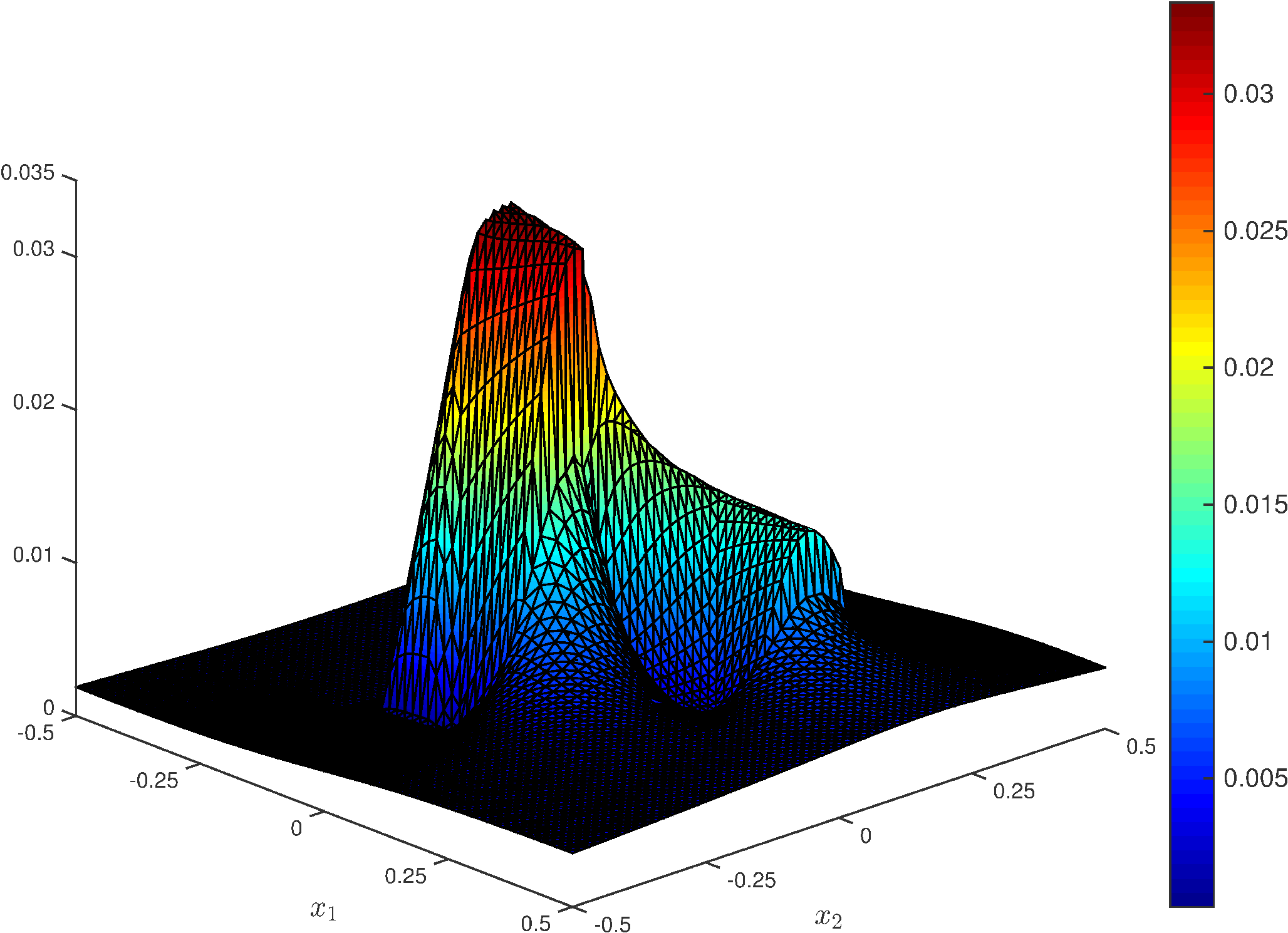}%
\end{center}
\caption{\label{fig:bsp3-sol}Interior and exterior solution
with full upwinding stabilization on
an uniformly generated mesh with $3072$ elements in the example in  
subsection~\ref{ex:bsp3}.}
\end{figure}
Our last example is a more practical problem.
The model can describe the stationary concentration of a 
chemical dissolved and
distributed in different fluids, where we have
a convection dominated problem in $\Omega$ and a diffusion distribution 
in $\Omega_e$. Note that the interior is a classical model problem and as described
above, the coupling with the exterior problem can `replace' the boundary condition, 
which might be difficult to find. 
Our interior domain $\Omega=(-1/4,1/4)^2\backslash \big([0,1/4]\times[-1/4,0]\big)$ 
is the classical L-shape. 
The diffusion matrix $\A=\alpha\,\I$ in $\Omega$ is piecewise constant and reads
\begin{align*}
  \alpha:\R\times\R\to\R:(x_1,x_2)\mapsto\begin{cases}\displaystyle
    10^{-7}& \text{for } x_2\leq 0,\\[1mm]
   \displaystyle 10^{-6}
    & \text{for } x_1>0,\\[1mm]
   \displaystyle 5\cdot 10^{-7}
    & \text{else}.
  \end{cases}
\end{align*}
Additionally, we choose  $\b=(15,10)^T$ and $\r=10^{-2}$. 
The source is in the lower square, i.e.
$f=5$ for $-0.2\leq x_1\leq -0.1$, $-0.2\leq x_2\leq -0.05$, and
$f=0$ elsewhere. 
We prescribe the jumps $u_0=0$ and $t_0=0$. Instead of a logarithmic radiation condition, 
we impose that $u=a_\infty+\OO(1/|x|)$ and $|x|\to \infty$ for an unknown $a_\infty\in\R$. An exterior solution of the Laplace equation satisfying this type of asymptotic behavior at infinity must have zero average of the 
normal derivative on $\Gamma$, see~\cite{Costabel:1985-1}. 
We must add $a_\infty$ to the representation formulas
for the exterior solution~\eqref{eq:repformula}
and~\eqref{eq:repformulaangle}, respectively, 
and~\eqref{eq:cal1} becomes
\begin{align*}
u_e|_\Gamma=(1/2+\KK)u_e|_\Gamma-\VV\phi+a_\infty.
\end{align*}
Thus we have an additional term  $\spe{\psi_h}{a_\infty}{\Gamma}$ 
on the left-hand side of~\eqref{eq2:fvm}
and an additional equation, which ensures $\spe{1}{\phi_h}{\Gamma}=0$ 
as the counterpart.  
We use the full upwind scheme, i.e.~\eqref{eq:weightfunctionfull}, 
for the approximation of the convection term and start with a mesh of
$12$ triangles.
This example is similar to the one in~\cite[Subsection 7.2]{Erath:2012-1} but
with a smaller diffusion. Note that the problem is highly convection dominated
and the analytical solution is unknown.
An interior solution without any stabilization is
plotted in Figure~\ref{fig:bsp3-solint}\subref{subfig:bsp3solnoupwind} and shows
strong oscillations.
In Figure~\ref{fig:bsp3-solint}\subref{subfig:bsp3contour} we see the contour
lines based on a solution generated on a mesh $\TT$ with $49152$ elements.
The transport is mainly from the source $f\not = 0$ in the left lower square
in the direction of the convection $\b$. We also can see the interaction
with the exterior domain, hence, the contour lines are circular.
In general, the solution of such a problem may have local phenomena such as
injection wells. As seen in Figure~\ref{fig:bsp3-sol} 
this leads to step layers on the boundary $(0,0)$ to $(0,-1/4)$, 
due to the convection in this direction and the different diffusion coefficient of the
interior and exterior problem.
Since we consider here a domain with a reentrant corner and
model data with jumps, it is well known that uniform mesh refinement 
can not guarantee
optimal convergence rates, i.e. $u\not \in H^2(\Omega)$.
An adaptive mesh refinement steered through a robust
a~posteriori estimator could lead to a more accurate solution as one
can find in a similar example for the FVM-BEM three field coupling approach
in~\cite{Erath:2013-1}.

\section{Conclusions}
We presented a new FVM-BEM coupling method based on the 
non-symmetric approach to solve
a transmission problem, i.e., a convection diffusion reaction problem in 
an interior domain coupled with a diffusion process in an unbounded exterior domain.
The resulting scheme maintains local flux conservation, also in the case when
an upwind scheme for convection dominated problems is used.
We showed ellipticity of the continuous and discrete system 
or for some model configurations the ellipticity
of their equivalent stabilized system. Additionally, we could improve
the theoretical elliptic constant from previous works.
Note that the stabilized FVM-BEM system was only used for
theoretical purposes.
This allowed us to show existence and uniqueness, convergence, and an a~priori
estimate.  
We stress that  for some critical
model configurations the assumptions on the data and regularity of the unknown solution
are weaker than for the comparable three field FVM-BEM coupling. 
Moreover, the non-symmetric approach has less
discrete unknowns and thus is computational cheaper.
Our work gives us a recipe
for the coupling of BEM with a non-Galerkin method like FVM.
Our theoretical results were confirmed by three numerical examples, which illustrate
the strength of the chosen method in terms of local flux conservation and convection
dominated problems.

\section*{Acknowledgements}
The third author was partially supported by NSF grant DMS 1216356.
\bibliographystyle{alpha}

\bibliography{literature}

\newcommand{\etalchar}[1]{$^{#1}$}
\begin{thebibliography}{AFF{\etalchar{+}}13}

\bibitem[AFF{\etalchar{+}}13]{wien}
Markus Aurada, Michael Feischl, Thomas F{\"u}hrer, Michael Karkulik,
  Jens~Markus Melenk, and Dirk Praetorius.
\newblock {Classical {FEM}-{BEM} coupling methods: nonlinearities,
  well-posedness, and adaptivity}.
\newblock {\em Comput. Mech.}, 51(4):399--419, 2013.

\bibitem[BJ79]{Brezzi:1979-1}
Franco Brezzi and Claes Johnson.
\newblock {On the coupling of boundary integral and finite element methods}.
\newblock {\em Calcolo}, 16(2):189--201, 1979.

\bibitem[BR87]{Bank:1987-1}
Randolph~E. Bank and Donald~J. Rose.
\newblock {Some Error Estimates for the Box Method}.
\newblock {\em SIAM J. Numer. Anal.}, 24(4):777--787, August 1987.

\bibitem[Cai91]{Cai:1991-1}
Zhiqiang Cai.
\newblock {On the finite volume element method}.
\newblock {\em Numer. Math.}, 58(7):713--735, 1991.

\bibitem[Cha02]{Chatzipantelidis:2002-1}
Panagiotis Chatzipantelidis.
\newblock {Finite volume methods for elliptic PDE's: a new approach}.
\newblock {\em M2AN Math. Model. Numer. Anal.}, 36(2):307--324, 2002.

\bibitem[Cia78]{Ciarlet:1978-book}
Philippe~G. Ciarlet.
\newblock {\em {The finite element method for elliptic problems}}.
\newblock North-Holland Publishing Co., Amsterdam, 1978.

\bibitem[Cos87]{Costabel:1987-1}
Martin Costabel.
\newblock {Symmetric methods for the coupling of finite elements and boundary
  elements}.
\newblock In {\em Boundary elements IX, Vol.\ 1 (Stuttgart, 1987)}, pages
  411--420. Comput. Mech., Southampton, 1987.

\bibitem[Cos88]{Costabel:1988-1}
Martin Costabel.
\newblock {Boundary integral operators on Lipschitz domains: elementary
  results}.
\newblock {\em SIAM J. Math. Anal.}, 19(3):613--626, 1988.

\bibitem[CS85]{Costabel:1985-1}
Martin Costabel and Ernst~P. Stephan.
\newblock {A direct boundary integral equation method for transmission
  problems}.
\newblock {\em J. Math. Anal. Appl.}, 106(2):367--413, 1985.

\bibitem[ELL02]{Ewing:2002-1}
Richard~E. Ewing, Tao Lin, and Yanping Lin.
\newblock {On the accuracy of the finite volume element method based on
  piecewise linear polynomials}.
\newblock {\em SIAM J. Numer. Anal.}, 39(6):1865--1888, 2002.

\bibitem[Era10]{Erath:2010-phd}
Christoph Erath.
\newblock {\em {Coupling of the Finite Volume Method and the Boundary Element
  Method - Theory, Analysis, and Numerics}}.
\newblock PhD thesis, University of Ulm, April 2010.

\bibitem[Era12]{Erath:2012-1}
Christoph Erath.
\newblock {Coupling of the finite volume element method and the boundary
  element method: an a priori convergence result}.
\newblock {\em SIAM J. Numer. Anal.}, 50(2):574--594, 2012.

\bibitem[Era13a]{Erath:2013-2}
Christoph Erath.
\newblock {A new conservative numerical scheme for flow problems on
  unstructured grids and unbounded domains}.
\newblock {\em J. Comput. Phys.}, 245:476--492, 2013.

\bibitem[Era13b]{Erath:2013-1}
Christoph Erath.
\newblock {A posteriori error estimates and adaptive mesh refinement for the
  coupling of the finite volume method and the boundary element method}.
\newblock {\em SIAM J. Numer. Anal.}, 51(3):1777--1804, 2013.

\bibitem[FFKP15]{Dirk:Ela}
Michael Feischl, Thomas F{\"u}hrer, Michael Karkulik, and Dirk Praetorius.
\newblock Stability of symmetric and nonsymmetric fem--bem couplings for
  nonlinear elasticity problems.
\newblock {\em Numer. Math.}, 130(2):199--223, 2015.

\bibitem[GHS12]{Sayas:relaxing}
Gabriel~N. Gatica, George~C. Hsiao, and Francisco-Javier Sayas.
\newblock {Relaxing the hypotheses of {B}ielak-{M}ac{C}amy's {BEM}-{FEM}
  coupling}.
\newblock {\em Numer. Math.}, 120(3):465--487, 2012.

\bibitem[Hac89]{Hackbusch:1989-1}
Wolfgang Hackbusch.
\newblock {On first and second order box schemes}.
\newblock {\em Computing}, 41(4):277--296, 1989.

\bibitem[HS15]{Heuer:DG}
Norbert Heuer and Francisco-Javier Sayas.
\newblock Analysis of a non-symmetric coupling of interior penalty {DG} and
  {BEM}.
\newblock {\em Math. Comp.}, 84(292):581--598, 2015.

\bibitem[JN80]{Johnson:1980-1}
Claes Johnson and Jean-Claude N{\'e}d{\'e}lec.
\newblock {On the coupling of boundary integral and finite element methods}.
\newblock {\em Math. Comp.}, 35(152):1063--1079, 1980.

\bibitem[McL00]{McLean:2000-book}
William McLean.
\newblock {\em {Strongly elliptic systems and boundary integral equations}}.
\newblock Cambridge University Press, Cambridge, 2000.

\bibitem[OS13]{Of:2013-1}
G{\"u}nther Of and Olaf Steinbach.
\newblock {Is the one-equation coupling of finite and boundary element methods
  always stable?}
\newblock {\em Z. Angew. Math. und Mech.}, 93(6-7):476--484, 2013.

\bibitem[OS14]{Of:Drwp}
G{\"u}nther Of and Olaf Steinbach.
\newblock On the ellipticity of coupled finite element and one-equation
  boundary element methods for boundary value problems.
\newblock {\em Numer. Math.}, 127(3):567--593, 2014.

\bibitem[RST96]{Roos:1996-book}
Hans-G{\"o}rg Roos, Martin Stynes, and Lutz Tobiska.
\newblock {\em {Numerical methods for singularly perturbed differential
  equations}}, volume~24.
\newblock Springer, Berlin, 1996.

\bibitem[Say09]{Sayas:2009-1}
Francisco-Javier Sayas.
\newblock {The validity of Johnson-N\'ed\'elec's BEM-FEM coupling on polygonal
  interfaces}.
\newblock {\em SIAM J. Numer. Anal.}, 47(5):3451--3463, 2009.

\bibitem[Say13]{Sayas:review}
Francisco-Javier Sayas.
\newblock {The Validity of Johnson--N{\'e}d{\'e}lec's BEM--FEM Coupling on
  Polygonal Interfaces}.
\newblock {\em SIAM Review}, 55(1):131--146, 2013.

\bibitem[Ste11]{Steinbach}
Olaf Steinbach.
\newblock {A note on the stable one-equation coupling of finite and boundary
  elements}.
\newblock {\em SIAM J. Numer. Anal.}, 49(4):1521--1531, 2011.

\bibitem[Ste13]{Steinbach:ela}
Olaf Steinbach.
\newblock On the stability of the non-symmetric {BEM}/{FEM} coupling in linear
  elasticity.
\newblock {\em Comput. Mech.}, 51(4):421--430, 2013.

\bibitem[SW01]{Steinbach:2001-1}
Olaf Steinbach and Wolfgang~L. Wendland.
\newblock {On C. Neumann's method for second-order elliptic systems in domains
  with non-smooth boundaries}.
\newblock {\em J. Math. Anal. Appl.}, 262(2):733--748, 2001.

\end{thebibliography}

\end{document}